 \documentclass[reqno, 10.5pt]{amsart}
\usepackage{amsmath}
\usepackage{amscd}
\usepackage{amssymb}
\numberwithin{equation}{section}

\newcommand{\na}{\nabla}

\newcommand{\bR}{{\Bbb R}}

\newcommand{\RR}{\mathbb{R}}

\newtheorem{theorem}{Theorem}[section]
\newtheorem{theorem/definition}{Theorem/Definition}[section]
\newtheorem{proposition}{Proposition}[section]
\newtheorem{lemma}{Lemma}[section]
\newtheorem{corollary}{Corollary}[section]
\theoremstyle{remark}
\newtheorem{remark}{Remark}[section]
\theoremstyle{definition}

\begin{document}
\title[Curvature estimates for 4D gradient Ricci solitons]
{On Curvature Estimates for four-dimensional gradient Ricci solitons}
\author{HUAI-DONG CAO}
\address{Department of Mathematics,  Lehigh University,
Bethlehem, PA 18015, USA}
\email{huc2@lehigh.edu}

\thanks{Research partially supported by a Simons Foundation Collaboration Grant (\#586694 HC)}

\begin{abstract} In this survey paper, we analyse and compare the recent curvature estimates for three types of $4$-dimensional gradient Ricci solitons, especially between Ricci shrinkers \cite{MW} and expanders \cite{CaoLiu}. In addition, we provide some new curvature estimates for $4$-dimensional gradient steady Ricci solitons, including the sharp curvature estimate $|Rm|\le C R$ for gradient steady Ricci solitons with positive Ricci curvature (see Theorem 1.1).

\end{abstract} 

\dedicatory{Dedicated to Professor Renato Tribuzy on the occasion of his 75th birthday.}
\maketitle
\date{}

\section{Introduction}

A complete Riemannian manifold $(M^n, g)$ is called a {\it gradient Ricci soliton} if there exists a smooth function $f$ on $M^n$ such that the Ricci tensor $Rc=\{R_{ij}\}$ of the metric $g=\{g_{ij}\}$ satisfies the equation $Rc +\na^2 f=\lambda g$  for some $\lambda \in{\mathbb R}$, or in local coordinates, $$R_{ij}+\nabla_i\nabla_jf=\lambda g_{ij} . \eqno(1.1)$$ Here, $\na ^2 f=\{\na_i\na_j f\}$ denotes the Hessian of $f$. The Ricci soliton $(M^n, g, f)$ is said to be  
{\it shrinking}, or {\it steady}, or  {\it expanding} if $\lambda>0$, or $\lambda=0$, or $\lambda<0$, respectively. The function $f$ is called a {\it potential function} of the gradient Ricci soliton. Clearly, when $f$ is a constant function  $g$ is simply an Einstein metric. Thus, gradient Ricci solitons are a natural extension of Einstein manifolds. They are also self-similar solutions to Hamilton's Ricci flow and often model formation of singularities \cite{Ha95F}, thus playing an important role in the study of the Ricci flow. In particular, shrinking and steady solitons arise as possible finite time Type I and Type II, respectively, singularity models in the Ricci flow on compact manifolds. We refer the readers to a survey article of the author \cite{Cao08b}, and the references therein, for a basic overview. 

There have been a lot of advances by experts in the field in investigating the geometry of gradient shrinking and steady Ricci solitons and their classifications in the past two decades, while the research activities on expanding solitons have also picked up in recent years. In particular, there is now a complete classification of 3-dimensional finite time singularity models. In the shrinking case, as a consequence of the famous Hamilton-Ivey curvature pinching theorem \cite{Ha95F, Ivey}, Ivey \cite{Ivey} observed that a compact 3-dimensional gradient shrinking soliton must be a finite quotient of the round sphere ${\mathbb S}^3$. Subsequently, Perelman \cite{P2} showed that complete non-flat noncollapsed gradient shrinking solitons with bounded and nonnegative curvature are finite quotients of either ${\mathbb S}^3$ or ${\mathbb S}^2\times \bR$ (see also Hamilton \cite{Ha95F}). Naber \cite{Naber} showed that gradient shrinking solitons with bounded curvature are necessarily noncollapsed. Finally, Cao-Chen-Zhu \cite{CCZ08}, together with  Ni-Wallach \cite{NW}, classified 3-dimensional complete gradient shrinking solitons (see also related work by Petersen-Wylie \cite{PW2}). For gradient steady solitons, Brendle \cite{Brendle1} proved that the only 3-dimensional non-flat noncollapsed gradient steady soliotn is the rotationally symmetric Bryant soliton \cite{Bryant} on $\bR^3$, as claimed by Perelman \cite{P2}. Note that such a steady soliton necessarily has positive sectional curvature by a result of B.-L. Chen \cite{BChen}. More recently, Brendle \cite{Brendle2} classified 3-dimensional complete noncollapsed ancient solutions and especially proved Perelman's conjecture that any 3-dimensional noncollapsed eternal solution with bounded positive sectional curvature must be a gradient steady soliton (hence the Bryant soliton). In conclusion, the works of Hamilton \cite {Ha95F}, Perelman \cite{P2}, and Brendle \cite{Brendle2}, prove that any 3-dimensional finite time singularity model for the Ricci flow on a compact 3-manifold is a finite quotient of either ${\mathbb S}^3/\Gamma$ or a round cylinder ${\mathbb S}^2 \times \RR$ or its ${\mathbb Z}_2$-quotient, or isometric to the Bryant soliton. 

One of the special features of the Ricci flow in dimension $n=3$ is that the positivity of the Ricci curvature (and the sectional curvature) is preserved by the Ricci flow \cite{Ha82, Shi1}. Especially, a magic of 3-dimensional Ricci flow is the Hamilton-Ivey curvature pinching property: whenever a solution $g(t)$ to the Ricci flow on a 3-manifold forms a finite time singularity, meaning $K(t)=\max_{x\in M^3} |Rm|(x, t)$ of the curvature tensor $Rm$ tends to the infinity as $t$ tends to the blow-up time $T<\infty$, the positive curvature is much larger than the (absolute value of) negative curvature; as a consequence, the limit of any sequence of parabolic rescaled solutions around almost maximal curvature points will converge to a noncollapsed ancient solution ${\tilde g}(t)$ to the Ricci flow with bounded and nonnegative curvature $0\le \widetilde {Rm}\le C$ (such a solution ${\tilde g}(t)$ is called an ancient $\kappa$-solution by Perelman \cite{P2}). Moreover, Chen \cite{BChen} proved that any complete 3-dimensional ancient solution $g(t)$, $-\infty < t<T$,  to the Ricci flow must have nonnegative sectional curvature (i.e., $Rm\ge 0$ when $n=3$). Since shrinking solitons and steady solitons are special ancient solutions, it follows that all 3-dimensional shrinking and steady Ricci solitons have nonnegative sectional curvature. In particular, their curvature tensor $Rm$ is controlled by the scalar curvature $R$, i.e., $|Rm|\le cR$ for some universal constant $c>0$. 

In general,  when studying complete noncompact Riemannian manifolds it is crucial to gain information on curvature control at infinity. Thus, it is important to understand the curvature behavior of higher dimensional gradient Ricci solitons. However, when dimension $n\ge 4$, the Ricci flow no longer preserves the positivity of the Ricci curvature or sectional curvature in general; see, e.g., Ni \cite{Ni04}, Knopf \cite{Knopf06}, M\'aximo \cite{Max11} and Bettiol-Krishnan \cite{BK19}. While gradient shrinking solitons enjoy some special geometric properties, such as the optimal asymptotic growth estimates for potential functions by Cao-Zhou \cite{CaoZhou}, the volume growth estimates by Cao-Zhou \cite{CaoZhou} and Munteanu-Wang \cite{MW12}, the positivity of scalar curvature by Chen \cite{BChen} and Pigola-Rimoldi-Setti  \cite{PRS}, and the quadratic decay lower bound for the scalar curvature by Chow-Lu-Yang \cite{CLY} etc, there exist examples of gradient shrinking (and steady) solitons that have mixed Ricci or sectional curvatures; see, e.g., Koiso \cite{Koiso} and Cao \cite{Cao94}, Wang-Zhu \cite{WZ04}, Feldman-Ilmanen-Knopf \cite{FIK}, and Dancer-Wang \cite{DW11}. Thus, one cannot expect any Hamilton-Ivey type pinching property or a result similar to Chen \cite{BChen} for the curvature tensor of a general n-dimensional ancient solution when $n\ge 4$. 

Therefore, it was rather surprising that Munteanu and Wang \cite{MW} were able to prove the curvature estimate $|Rm|\le C R$ for 4-dimensional complete (noncompact) gradient shrinking solitons with bounded scalar curvature $R\le R_0$. Their curvature estimate, together with the uniqueness result of Kotschwar-Wang \cite{KW15}, has played a crucial role in the recent progress of classifying 4-dimensional complete noncompact gradient shrinking Ricci solitons, as well as in the classification of complex 2-dimensional complete gradient K\"ahler-Ricci solitons with scalar curvature going to zero at infinity by Conlon-Deruelle-Sun \cite{CDS19}. Subsequently, inspired by the work of Munteanu and Wang \cite{MW}, Cao-Cui \cite{CaoCui} and Chan \cite{Chan1} studied the corresponding curvature estimates for steady solitons either with positive Ricci curvature or with scalar curvature decaying to zero at infinity; see also a recent extension of the curvature estimate of Munteanu and Wang to 4-dimensional gradient shrinking solitons with suitable scalar curvature growth  by Cao-Ribeiro-Zhou \cite{CRZ}, and the work of Chow-Freedman-Shin-Zhang \cite{Chow et al} on curvature estimates for 4-dimensional gradient Ricci soliton singularity models. 

While expanding solitons are abundant and in general behave rather differently than shrinking solitons, there are some similarities between gradient expanding solitons with nonnegative Ricci curvature and gradient shrinking solitons, such as the optimal asymptotic growth estimates for potential functions (see, e.g.,  \cite{Cao et al, CarNi}) and the maximal volume growth estimate by Hamilton \cite{Ha05} (see also \cite{CLN} and its extension by Carrillo-Ni \cite{CarNi}). 
Moreover, recently P.-Y. Chan \cite{Chan2} proved that if $(M^4, g, f)$ is a complete noncompact gradient expanding Ricci soliton with bounded scalar curvature $|R|\le R_0$ and proper potential function $f$ so that $\lim_{r(x)\to \infty} f(x) =-\infty$, then the curvature tensor $Rm$ is bounded. Note that $Rc\ge 0$ implies $0\le R\le R_0$, for some $R_0>0$, and $f$ proper (see Propositions 4.1-4.2), hence it follows that $4$-dimensional complete gradient expanding solitons with nonnegative Ricci curvature $Rc\geq 0$  must have bounded Riemann curvature tensor $|Rm|\le C$.  Motivated by the work of Munteanu-Wang \cite{MW}, as well as Cao-Cui \cite{CaoCui} and Chan \cite{Chan1}, it is then natural to ask if one could also control the Riemann curvature tensor $Rm$ of a $4$-dimensional complete gradient expanding soliton with nonnegative Ricci curvature by its scalar curvature $R$. Despite some key differences with the shrinking case, notably the lack of scalar curvature quadratic (or polynomial) decay lower bound for expanders,  
this turns out to be possible as shown in the very recent work of T. Liu and the author \cite{CaoLiu} (see also Theorems 4.1 and 4.2). 

The main purpose of this article is to analyse and compare the curvature estimates for the three types of 4-dimensional  gradient Ricci solitons, especially between Ricci shrinkers and expanders. We hope this will be useful in understanding similarities and differences of the three types of gradient Ricci solitons in dimension four and also shed some light to the higher dimensional case.  

Moreover, we obtained some new curvature estimates for 4-dimensional gradient steady solitons, including the following sharp curvature estimate $|Rm|\le C R$ for  gradient steady Ricci solitons with positive Ricci curvature (see also Theorem 5.2). 

\begin{theorem}
Let $(M^4, g, f)$ be a complete noncompact $4$-dimensional
gradient steady Ricci soliton with positive Ricci curvature $Rc>0$ such that the scalar curvature $R$ attains its maximum at some point $x_0\in M$. Then, 
\[ |Rm|  \le C R \quad \mbox{on} \ M. \]
\end{theorem}

\vfill\eject

\section{Preliminaries}

In this section, we recall some basic facts and collect several known differential identities and inequalities on curvatures of gradient Ricci solitons satisfing the equation
\[ R_{ij}+\nabla_i\nabla_jf=\lambda g_{ij} . \]
Moreover, we recall a key estimate of the curvature tensor $Rm$ in terms of $Rc$, $\na Rc$, and the potential function $f$ for 4-dimensional gradient Ricci solitons due to Munteanu and Wang \cite{MW}.  

Throughout the paper, we denote by $$Rm=\{R_{ijkl}\}, \quad Rc=\{R_{ik}\},\quad R $$ the Riemann curvature tensor, the Ricci tensor, and the scalar curvature of the metric $g=\{g_{ij}\}$, respectively.

\begin{lemma} {\bf (Hamilton \cite{Ha95F})} Let $(M^n, g, f)$
be a  gradient Ricci soliton satisfying Eq. (1.1).
Then
$$ R+\Delta f =n \lambda, \eqno (2.1)$$
$$\nabla_iR=2R_{ij}\nabla_jf, \eqno(2.2)$$
$$R+|\nabla f|^2=2\lambda f +C_0 \eqno(2.3)$$ for some constant $C_0$.
\end{lemma}

Furthermore, in the shrinking or expanding case (i.e., $\lambda\neq 0$), we can normalize the potential function $f$ in (2.3) so that 
 \[R+|\nabla f|^2=2\lambda f. \eqno(2.4)\]

We now collect several well-known differential identities for the curvatures $R, Rc$ and $Rm$ that we shall use later. 

\begin{lemma} Let $(M^n, g, f)$ be a gradient Ricci soliton satisfying Eq. (1.1). Then, 
\begin{eqnarray*}
\Delta_{f} R &=& 2\lambda R-2|Rc|^2,\\
\Delta_{f} R_{ik} &=&  2\lambda R_{ik} -2R_{ijkl}R_{jl},\\
\Delta_{f} {Rm} &=&   2\lambda Rm+ Rm\ast Rm,\\
\na_lR_{ijkl} &=& \na_jR_{ik}-\na_i R_{jk}=R_{ijkl}\na_lf, 
\end{eqnarray*}
where,  on the RHS of the third equation, $Rm\ast Rm$ denotes the sum of a finite number of terms involving quadratics in $Rm$, and $\Delta_f =:\Delta -\nabla f\cdot \nabla$ is the weighted Laplace operator, which is self-adjoint with respect to the weighted measure $e^{-f}dV_g.$
\end{lemma}

It is also useful to note the following differential identity,
\begin{equation*}
\Delta_f (R^{-a}) =aR^{-a}\left(-2\lambda +2R^{-1}|Rc|^2 +(a+1)|\na\ln R|^2 \right), \eqno(2.5)
\end{equation*}
which will be used frequently later. 

Moreover, based on Lemma 2.2,  one can easily derive the following differential inequalities  (see also \cite{MW, CaoCui, CaoLiu}):

\begin{lemma} Let $(M^n, g, f)$ be a gradient  Ricci soliton satisfying Eq. (1.1). Then
\begin{eqnarray*}
\Delta_{f} |Rc|^2 & \ge & 2|\na Rc|^2 + 4\lambda |Rc|^2-4|Rm| |Rc|^2, \\
\Delta_{f}|Rm|^2  &\ge & 2|\na Rm|^2  + 4\lambda |Rm|^2-c|Rm|^3,\\
\Delta_{f} |Rm| &\ge & 2\lambda |Rm|-c|Rm|^2.
\end{eqnarray*}
Here $c>0$ is some universal constant depending only on the dimension $n$.
\end{lemma}

\begin{remark} To derive the third  inequality, one needs to use Kato's inequality
$|\nabla |Rm||\le |\nabla Rm|$.
\end{remark}

We also  have the following differential inequalities on the covariant derivative $\na Rm$ of the curvature tensor (also see \cite{MW} for the shrinking case and \cite{CaoLiu} for the expanding case). 

\begin{lemma} Let $(M^n, g, f)$ be a gradient Ricci soliton satisfying Eq. (1.1). Then
\begin{eqnarray*}
\Delta_{f} |\na Rm|^2 &\ge & 2|\na^2 Rm|^2  +6 \lambda |\na Rm|^2-c|Rm| |\na Rm|^2 \quad \mbox{and} \\
\Delta_{f} |\na Rm| &\ge & 3\lambda |\na Rm|-c|Rm||\na Rm|,
\end{eqnarray*}
where $c>0$ is some universal constant depending only on the dimension $n$.
\end{lemma}

\begin{proof} First of all, by commuting covariant differentiations and Lemma 2.2, 
\begin{eqnarray*} 
\Delta_{f} (\na_q Rm )& = & \na_p\na_p (\na_q Rm)-\na_p(\na_q Rm)\cdot\na_p f \\
&=& \na_p \big(\na_q \na_p Rm +Rm\ast Rm\big)-\big(\na_q\na_p Rm+R_{pq\bullet\bullet}Rm\big)\cdot\na_p f\\
&=&   \na_q (\Delta  Rm) -\na_q(\na_p Rm \na_p f)+(\na_p Rm) (\na_q\na_p f) +Rm\ast \na_q Rm \\
&=&  \na_q (\Delta_f  Rm) +\lambda \na_q Rm +Rm\ast \na_q Rm\\
&=& 3\lambda \na_q Rm +Rm\ast \na_q Rm,
\end{eqnarray*}
where we have used the 4th identity in Lemma 2.2  in the third equality and the soliton equation (1.1) in the 4th equality.  Thus, 
\begin{eqnarray*} 
\Delta_{f} |\na Rm|^2 & = & 2|\na^2 Rm|^2 +2 (\na Rm) \Delta_{f} (\na Rm)\\
 &\ge & 2|\na^2 Rm|^2 +6\lambda  |\na Rm|^2 -c|Rm|  |\na Rm|^2.
\end{eqnarray*}
This proves the first inequality. The second inequality follows easily from the first one and Kato's inequality. 
\end{proof}

Finally, we need the following key observation for $4$-dimensional gradient Ricci solitons due to Munteanu  and Wang \cite{MW} (see also Lemma 1 in \cite{Chan1}).

\begin{lemma}{\bf(Munteanu-Wang)} 
\label{Curv} Let $(M^4, g, f)$ be a $4$-dimensional gradient Ricci soliton.
Then, there exists some universal constant $A_0>0$ such that, at any point where $\na f\neq 0$, 
$$ |Rm|\le A_0\left(|Rc|+ \frac {|\nabla Rc|} {|\nabla f|} \right). $$
\end{lemma}

\begin{proof} Suppose $\na f\neq 0$ at a point $p\in M$. If we choose an orthonormal basis $\{E_1, \cdots, E_4\}$ at $p$ with $E_4=\nabla f/|\nabla f|$, then 
 $$R_{1212}=\frac 12 \big( R_{11}+R_{22}-R_{33}-R_{44}\big) +R_{3434}, $$ 
$$R_{1313}=\frac 12 \big( R_{11}+R_{33}-R_{22}-R_{44}\big) +R_{2424}, $$ 
$$R_{2323}=\frac 12 \big( R_{22}+R_{33}-R_{11}-R_{44}\big) +R_{1414} .$$ 
Moreover, by Lemma 2.2, 
 $$|R_{ijk4}| |\nabla f| =|\nabla_4R_{ijk4}|=|\nabla_jR_{ik}-\nabla_iR_{jk}|\leq 2 |\na Rc|.$$
Therefore, all sectional curvatures $K_p$ at $p$ are bounded by  
$$|K_p| \leq 2|Rc| (p)+ 2\frac {|\nabla Rc|} {|\nabla f|}(p).$$
This completes the proof of Lemma 2.5. 

\end{proof}

\section{Curvature estimates for 4D gradient shrinking Ricci solitons}

Gradient shrinking Ricci solitons are important geometric objects in understanding Hamilton's Ricci flow. They typically arise as Type I singularity models and play a crucial role in the singularity analysis of the Ricci flow. Indeed, it  is known by the works of Naber \cite{Naber} and Enders-M\"uller-Topping \cite{EMT} that the blow-ups around any finite time Type I singularity point of a solution to the Ricci flow on a compact manifold  converge to a (nontrivial) gradient shrinking Ricci soliton.

By scaling, we may assume $\lambda=\frac1 2$ in (1.1) so that the gradient shrinking Ricci soliton equation becomes
\begin{equation}
Rc+\na^2 f =\frac 1 2 g \quad \text{or} \quad R_{ij}+\nabla_i\nabla_jf=\frac 1 2  g_{ij}. 
\end{equation}

Note that, by (2.1) and (2.4), we have 
\begin{equation}
 |\nabla f|^2=f-R 
\end{equation}  
and   
\begin{equation}
 \Delta_f f=-f+\frac n 2   
\end{equation}
where $\Delta_f =:\Delta -\nabla f\cdot \nabla$ is the weighted Laplace operator, which is self-adjoint with respect to the weighted measure $e^{-f}dV_g.$

Also, by Lemmas 2.2-2.4, we have the following differential identities and inequalities for gradient shrinking Ricci solitons satisfying Eq. (3.1).
\begin{eqnarray}
\Delta_{f} R &=& R-2|Rc|^2,  \notag\\
\Delta_{f} R_{ik} &=&R_{ik} -2R_{ijkl}R_{jl},  \notag \\
\Delta_{f} {Rm} &=&  Rm+ Rm\ast Rm  \notag,\\
\na_lR_{ijkl} &=& \na_jR_{ik}-\na_i R_{jk}=R_{ijkl}\na_lf,  \notag \\
\Delta_{f} |Rc|^2 & \ge & 2|\na Rc|^2 +2|Rc|^2-4|Rm| |Rc|^2,  \label{id} \\
\Delta_{f}|Rm|^2  &\ge & 2|\na Rm|^2 + 2|Rm|^2-c|Rm|^3,  \notag \\
\Delta_{f} |Rm| &\ge & |Rm|-c|Rm|^2  \notag \\
\Delta_{f} |\na Rm|^2 &\ge & 2|\na^2 Rm|^2 + 3 |\na Rm|^2-c|Rm| |\na Rm|^2,  \notag \\
\Delta_{f} |\na Rm| &\ge & \frac 32 |\na Rm|-c|Rm||\na Rm|.  \notag
\end{eqnarray}
Here, $c>0$ is some universal constant depending only on the dimension $n$.

Next, let us recall several important geometric properties of general gradient shrinking Ricci solitons that will be useful in the curvature estimates.

\begin{proposition} {\bf (Cao-Zhou \cite{CaoZhou})}  Let $(M^n, g , f)$ be a complete noncompact gradient shrinking Ricci soliton satisfying (3.1). Then the potential function $f$ satisfies the estimates
\begin{equation}
\frac 1 4(r(x)-c_1)^2\le f(x)\le \frac 1 4 (r(x)+2\sqrt{-f(x_0)})^2,   
\end{equation}
where $r(x)$ is the distance function from any fixed base point $x_0$ in $M$. 
\end{proposition}

\begin{proposition} {\bf (Chen \cite{BChen})}  Let $(M^n, g , f)$ be a complete noncompact gradient shrinking Ricci soliton satisfying (3.1). Then the scalar curvature is nonnegative, 
\[ R\geq 0 \quad  \text{ \ on } M\text{.} \]
\end{proposition}

\begin{remark}
Moreover, $R>0$ on $M$ unless $(M^n, g , f)$ is the Gaussian shrinking soliton (see  Pigola-Rimoldi-Setti \cite{PRS}). 
\end{remark}

\begin{proposition} {\bf (Chow-Lu-Yang \cite {CLY})} 
Let $(M^n, g, f)$ be a complete noncompact gradient shrinking Ricci soliton. 
Then there exists a constant $C>0$ such that 
\begin{equation}
R\geq  C/f \text{.}
\end{equation}
In particular, by Proposition 3.1,  $R$ has at most quadratic decay. 
\end{proposition}

Now we can state the main curvature estimates  for 4-dimensional gradient shrinking Ricci solitons due to Munteanu and Wang \cite{MW}. 

\begin{theorem} {\bf (Munteanu-Wang \cite{MW})} 
Let $\left( M^4, g, f\right) $ be a $4$-dimensional complete noncompact gradient shrinking Ricci soliton with bounded scalar curvature $R\le R_0$. Then there exists
a constant $C>0$ such that 
\begin{equation*}
\left\vert \mathrm{Rm}\right\vert \leq C\, R \quad {\mbox{and}} \quad 
\left\vert \mathrm{\na Rm}\right\vert \leq C\, R \quad  \text{ \ on } M\text{.}
\end{equation*}
\end{theorem}

In order to compare the similarities and the differences among shrinking, expanding, and steady solitons in the $4$-dimensional curvature estimates, we shall present here a sketch of the proof of Theorem 3.1 by Munteanu and Wang (taking a good amount of detailed arguments from \cite{MW} in the relevant parts of the proof), and refer the interested readers to their original paper \cite{MW} for a complete proof.  

\medskip
\noindent
{\it Sketch of the Proof of Theorem 3.1 \!.} \ We denote by
\begin{equation}
D(t):= \{x\in M : f(x) \leq t \}. 
\end{equation}
By (3.2), Proposition 3.1, and the scalar curvature assumption $R\leq R_0$, it is obvious that there exists $r_{0}>0$, depending only
on $R_0$, such that 
\begin{equation}
\left\vert \nabla f\right\vert \geq \frac{1}{2}\sqrt{f}\geq 1\text{ \ on \ }%
M\backslash D\left( r_{0}\right) . 
\end{equation}%
The entire proof consists of a series of elaborate pointwise estimates on $Rc$, $Rm$, $\na Rm$, as well as a gradient estimate on the scalar curvature $R$.

\medskip
{\bf Step 1} ({\it Initial $|Rc|$ estimate}).  There exists  a constant $C>0$ such that
\begin{equation}
|Rc|^2 \leq C R^{1/2} \quad \text{on} \ M . \label{Rc}
\end{equation} 

First of all, note that by  (\ref{id}), Lemma \ref{Curv} and  (3.8), we have
\begin{eqnarray}
\Delta _{f}\left\vert \mathrm{Rc}\right\vert ^{2} &\geq &2\left\vert \nabla 
\mathrm{Rc}\right\vert ^{2}-c\left\vert \mathrm{Rm}\right\vert \left\vert 
\mathrm{Rc}\right\vert ^{2}  \label{F1} \\
&\geq &2\left\vert \nabla \mathrm{Rc}\right\vert ^{2}-\frac{c}{\sqrt{f}}%
\left\vert \nabla \mathrm{Rc}\right\vert \left\vert \mathrm{Rc}\right\vert
^{2}   -c\left\vert \mathrm{Rc%
}\right\vert ^{3}.  \notag
\end{eqnarray}%
For any $a \in (0, 1)$, set  $u:=|Rc|^2/R^{a}$. Then direct computations  and using (2.5) and (3.8) yield the following differential inequality, 
\begin{equation}
\Delta _{f}u\geq \left( 2a-\frac{c}{1-a}\frac{R}{ f}\right)
u^{2}R^{a-1}-c\,u^{\frac{3}{2}}\,R^{\frac{a}{2}}-c\,u,  \quad \text{on}\  M\backslash D\left(r_{0}\right) . 
\end{equation}%
 Here, $c>0$ is a universal constant. 
Now, setting $a=1/2$, then 
\begin{eqnarray*}
\Delta _{f}u & \geq & \left( 1-C/{f}\right)
u^{2}R^{-1/2}-C\,u^{\frac{3}{2}}\,-C\,u \\
& \geq & \frac 1 2 u^{2}R^{-1/2}-C\,u^{\frac{3}{2}}\,-C\,u  \quad \text{on} \ M\backslash D\left( r_{1}\right), 
\end{eqnarray*}
where $u=R^{-1/2}|Rc|^2$, $C> 0$ depends only on $R_0$, and $r_1>r_0>0$ is chosen so that 
$$ 1-C/f\geq 1/2  \quad \text{on}\ M\backslash D\left( r_{1}\right),$$
which is possible by Proposition 3.1.

Next, let $\varphi (t)$ be a nonnegative  smooth  function on $\mathbb R^{+}$  defined by
$$ {\varphi (t) =\left\{
       \begin{array}{ll}
  1, \ \ \quad  \rho\le t\le 2\rho ,\\[4mm]
    0, \ \ \quad  0\le t\le \rho/2 \ \ \mbox{and} \ \ t\ge 3\rho,
       \end{array}
    \right.}$$
such that 
\begin{equation}
 t^2 \left(|\varphi'(t)|^2+|\varphi''(t)|\right)\le c 
\end{equation}%
for some universal constant $c$ and arbitrarily large $\rho>2r_1$. Now take $\varphi=\varphi (f(x))$ as a cut-off function, with support in $D(3\rho)\backslash D\left(\rho/2\right)$. Note that 
\begin{equation}
 |\nabla \varphi|=|\varphi'\nabla f|\leq  \frac {c} {\rho} |\na f| \leq  \frac {c} {\sqrt{\rho}} \qquad  \mbox{and} 
\end{equation}%
\begin{equation}
|\Delta_f (\varphi)| =|\varphi' \Delta_f f +\varphi''|\nabla f|^2|\le \frac {c} {\rho}f +  \frac {c} {\rho^2}|\na f|^2\le 2c 
\end{equation}%
on $D(3\rho)\backslash D(\rho/2)$.

Set $G=\varphi^2 u$. Then direct computations yield that
\begin{eqnarray*} \varphi^2 \Delta_f G 
&\ge & \frac 1 2 G^2R^{-1/2}  -C G ^{3/2}-C G +2\nabla G\cdot\nabla \varphi^2.
\end{eqnarray*}
Now it follows from the standard maximum principle argument that  $G\le C$ on $D(3\rho)$ for some constant $C>0$  depending on  $R_0$, but independent of $\rho$. Hence, on $D(2\rho)\backslash D(\rho)$,  
$$ R^{-1/2} |Rc|^2 =G \leq C.$$
Since $\rho>0$ is arbitrarily large, we get 
$$  |Rc|^2\le C R^{1/2} \le CR^{1/2}_0  \quad \text{on} \ M .$$

\begin{remark} By the same argument, one can actually show that 
\begin{equation}%
|Rc|^2 \leq C_a R^{a} \quad \text{on} \ M , 
\end{equation}%
for any $a\in (0, 1)$, with $C_a \to \infty$ as $a\to 1$.
\end{remark}

\medskip
{\bf Step 2} ({\it Initial $|Rm|$ and $|\na Rm|$ estimates}). There exists $C>0$  such that 
\begin{equation}
 \left\vert \mathrm{Rm}\right\vert +\left\vert \nabla \mathrm{Rm}\right\vert  \leq C \quad \text{on} \ M.
\end{equation}%

Using (\ref{id}), one sees that 
\begin{equation*}
\Delta _{f}\left\vert \mathrm{Rm}\right\vert \geq -c\left\vert \mathrm{Rm}%
\right\vert ^{2}.
\end{equation*}%
Rewrite this as 
\begin{equation}
\Delta _{f}\left\vert \mathrm{Rm}\right\vert \geq \left\vert \mathrm{Rm}%
\right\vert ^{2}-\left( c+1\right) \left\vert \mathrm{Rm}\right\vert ^{2}.
\label{R1}
\end{equation}%
By Lemma \ref{Curv} and (\ref{Rc}), we have
\begin{equation*}
\left\vert \mathrm{Rm}\right\vert ^{2}\leq C\left( \frac{1}{f}\left\vert
\nabla \mathrm{Rc}\right\vert ^{2}+1\right) .
\end{equation*}%
Plugging into (\ref{R1}), we get

\begin{equation}
\Delta _{f}\left\vert \mathrm{Rm}\right\vert \geq \left\vert \mathrm{Rm}%
\right\vert ^{2}-\frac{C}{f}\left\vert \nabla \mathrm{Rc}\right\vert ^{2}-C.
\label{R2}
\end{equation}%
On the other hand, we know from (\ref{F1}) and (\ref{Rc}) that 
\begin{equation}
\Delta _{f}\left\vert \mathrm{Rc}\right\vert ^{2}\geq \left\vert \nabla 
\mathrm{Rc}\right\vert ^{2}-C.  \label{R3}
\end{equation}%
Therefore, combining (\ref{R2}) and (\ref{R3}), we obtain
\begin{eqnarray*}
\Delta _{f}\left( \left\vert \mathrm{Rm}\right\vert +\left\vert \mathrm{Rc}%
\right\vert ^{2}\right) 
&\geq &\frac{1}{2}\left( \left\vert \mathrm{Rm}\right\vert +\left\vert 
\mathrm{Rc}\right\vert ^{2}\right) ^{2}-C.
\end{eqnarray*}%
That is, the function 
\begin{equation*}
v:=\left\vert \mathrm{Rm}\right\vert +\left\vert \mathrm{Rc}\right\vert ^{2}
\end{equation*}%
satisfies the differential inequality 
\begin{equation*}
\Delta _{f}v\geq \frac{1}{2}v^{2}-C \quad \text{on} \ M\backslash D\left(
r_{0}\right).
\end{equation*}%
By applying the same cut-off function and the maximum principle argument as in Step 1 above, we conclude that $v\le C$ on $M\backslash D\left( r_{0}\right) .$ This implies 
\begin{equation}
\sup_{M}\left\vert \mathrm{Rm}\right\vert \leq C\text{.}  \label{R-1}
\end{equation}%

Next, using (3.20) and (\ref{id}), we have 
\begin{equation*}
\Delta _{f}\left\vert \nabla \mathrm{Rm}\right\vert \geq -C\left\vert \nabla 
\mathrm{Rm}\right\vert .
\end{equation*}%
We also know from (\ref{id}) that 
\[ \Delta _{f}\left\vert \mathrm{Rm}\right\vert ^{2} \geq 2\left\vert \nabla 
\mathrm{Rm}\right\vert ^{2}-c\left\vert \mathrm{Rm}\right\vert ^{3} 
\geq 2\left\vert \nabla \mathrm{Rm}\right\vert ^{2}-C\text{. } \]
Hence,
\begin{equation*}
\Delta _{f}\left( \left\vert \nabla \mathrm{Rm}\right\vert +\left\vert 
\mathrm{Rm}\right\vert ^{2}\right) \geq \left( \left\vert \nabla \mathrm{Rm}%
\right\vert +\left\vert \mathrm{Rm}\right\vert ^{2}\right) ^{2}-C\text{. }
\end{equation*}%
Now a maximum principle argument as above shows that $\left\vert \nabla 
\mathrm{Rm}\right\vert +\left\vert \mathrm{Rm}\right\vert ^{2}$ is bounded
on $M.$ So 
\begin{equation*}
\left\vert \nabla \mathrm{Rm}\right\vert \leq C  \label{T2}
\end{equation*}%
for some constant $C>0$.

\medskip
{\bf Step 3} ({\it Improved $|Rc|$ estimate}). There exists $C>0$ such that 
\begin{equation}
|Rc|^2 \leq C R \quad \text{on} \ M .
\end{equation}

For $u=|Rc|^2/R^{a}$, let us recall the following differential inequality from (3.11):  
\begin{equation*}
\Delta _{f}u\geq \left( 2a-\frac{c}{1-a}\frac{R}{f}\right)
u^{2}R^{a-1}-c\,u^{\frac{3}{2}}\,R^{\frac{a}{2}}-c\,u,  \quad \text{on}\  M\backslash D\left(r_{0}\right). 
\end{equation*}%
Now, we take the same cut-off function $\varphi$ as in Step 1 above and set $G=\varphi^2 u$.  Then, by direct computations, we have 
\begin{eqnarray*} \varphi^2 \Delta_f G 
&\ge &  \left( 2a-\frac{c}{1-a} \frac {R}{f}\right) G^2R^{a-1}  -C G ^{3/2}-C G +2\nabla G\cdot\nabla \varphi^2.
\end{eqnarray*}
Since $\varphi$ has support in $D(3\rho)\backslash D\left(\rho/2\right)$, we know that $f\ge \rho/2 $ on the support of $\varphi$. Hence, one may choose $a=1-C/\rho$ for $C>0$ sufficiently large so that 
$$ 2a-\frac{c}{1-a} \frac {R}{f}\geq 1  \quad \text{on}\ D(3\rho)\backslash D\left(\rho/2\right).$$
Consequently, 
\begin{eqnarray*} \varphi^2 \Delta_f G & \geq &  G^2R^{a-1}  -C G ^{3/2}-C G +2\nabla G\cdot\nabla \varphi^2.
\end{eqnarray*}
Since $a<1$ and $R^{a-1}\geq R_0^{a-1}$, it follows from the standard maximum principle that $G\leq C$ on 
$D(2\rho)\backslash D(\rho)$. Hence, on $D(2\rho)\backslash D(\rho)$, 
\[ R^{-1} |Rc|^2= G R^{a-1} \le C R^{a-1}.\]
On the other hand, by Proposition 3.3,  $R\ge C/f$ on $M$.  
Thus, since $a-1=-C/\rho$ and $R\ge C/\rho$ on $D(2\rho)$, it follows that $R^{a-1} \le C$ for some $C>0$. Therefore, 
$$|Rc|^2 \leq C R \quad \text{on} \ M.$$

\medskip
{\bf Step 4} ({\it Improved $|Rm|$ estimate}). There exists $C>0$ such that 
\begin{equation}
|Rm|^2 \leq C R \quad \text{on} \ M .
\end{equation}

This follows immediately from Lemma 2.5,  the estimate $|Rc|^2\le CR$ in Step 3, the estimate $|\na Rm|\le C$ in Step 2 and the scalar curvature lower bound  $R\ge C/f$.  

\medskip
{\bf Step 5} ({\it Improved $|\na Rm|$ estimate}). There exists $C>0$ such that 
\begin{equation}
|\na Rm|^2 \leq C R \quad \text{on} \ M .
\end{equation}

By (\ref{id})  and direct computations, one obtains 
\[ \Delta _{f}\left( \left\vert \nabla \mathrm{Rm}\right\vert ^{2}R^{-1}\right) 
\geq  2\left\vert \nabla \mathrm{Rm}\right\vert ^{2}R^{-1}-c\left\vert 
\mathrm{Rm}\right\vert \left\vert \nabla \mathrm{Rm}\right\vert ^{2}R^{-1}. \]
On the other hand, using the estimates $|Rm|^2\le CR$ in Step 4 and $|\na Rm|\le C$  in Step 2, one gets
\[ c\left\vert \mathrm{Rm}\right\vert \left\vert \nabla \mathrm{Rm}\right\vert
^{2}R^{-1} \leq  C\left\vert \nabla \mathrm{Rm}\right\vert R^{-\frac{1}{2}}
\leq \left\vert \nabla \mathrm{Rm}\right\vert ^{2}R^{-1}+C.\]
Thus, the function 
\begin{equation*}
w:=\left\vert \nabla \mathrm{Rm}\right\vert ^{2}R^{-1}-C
\end{equation*}%
satisfies 
\begin{equation}
\Delta _{f}w\geq w.  \label{A5}
\end{equation}%
Note that the RHS of the above differential inequality (3.24) is only linear in $w$, so the standard maximum principle argument as in Step 1 or Step 2 does not apply. Nonetheless, Munteanu and Wang were able to adapt the maximum principle argument to achieve the desired estimate (3.23) by making a clever use of the scalar curvature lower bound $R\ge C/f$ due to Chow-Lu-Yang \cite{CLY} as follows. 
Consider 
$$ {\psi(t) =\left\{
       \begin{array}{ll}
  1-t/{\rho}, \ \ \ \ \! 0\le t\le \rho ,\\[4mm]
    0, \ \ \ \!\qquad  \quad t\ge \rho. 
       \end{array}
    \right.}$$
Then the cut-off function $\psi \left( f\right) $ on $M$
satisfies 
\begin{equation}
\left\vert \nabla \psi \right\vert =\frac{\left\vert \nabla f\right\vert }{%
\rho}  \label{A5'} \quad \text{and} \quad \Delta _{f}\psi =\frac{1}{\rho}\left( f-2\right) . 
\end{equation}
Therefore, for $G:=\psi ^{2}\,w,$ using (\ref{A5}), we have
\begin{equation}
\psi^2\Delta _{f}G\geq \left( \psi^2 +\psi \frac{2}{\rho}\left( f-2\right) -6\left\vert \nabla \psi \right\vert ^{2}\right) G+2
\left\langle \nabla G,\nabla \psi^2 \right\rangle .  \label{A6}
\end{equation}%

{\bf Case 1}: $G\left( q\right) <0$ at the maximum point $q$ of $G.$ Then $w\le 0$ on $D\left( R\right)$, which 
implies that 
\begin{equation*}
\sup_{D\left( \rho/{2}\right) }\left( \left\vert \nabla \mathrm{Rm}%
\right\vert ^{2}R^{-1}\right) \leq C.
\end{equation*}%

{\bf Case 2}:  $G\left( q\right) >0$ at the maximum point $q$ of $G.$ Then (%
\ref{A6}) implies that 
\begin{equation}
\frac{2}{\rho}\left( f-2\right) \psi \leq 6\left\vert \nabla \psi \right\vert
^{2}\leq 6\frac{1}{\rho^{2}}f.  \label{A7}
\end{equation}

If $q\in D\left( r_{0}\right) $, then 
\begin{eqnarray*}
\sup_{D\left( {\rho}/{2}\right) }\left( \left\vert \nabla \mathrm{Rm}%
\right\vert ^{2}R^{-1}\right) &\leq & C+4\,\sup_{D\left({\rho}/{2}\right)
}G \\
& \leq &  C+4\,\sup_{D\left( r_{0}\right) }G \  \leq  \  C.
\end{eqnarray*}%

If $q\in M\backslash D\left( r_{0}\right)$, then $%
f\left( q\right) -2\geq \frac{1}{2}f\left( q\right)$.  By (\ref{A7}), $\psi
\left( q\right)\leq 6/ \rho.$ This shows that $f\left( q\right) \geq \rho-6.$
Hence, since $|\na Rm|\le C$ and $Rf\ge C >0$, 
\begin{eqnarray*}
\frac{1}{4}\sup_{D\left( \rho/{2}\right) }\left( \left\vert \nabla 
\mathrm{Rm}\right\vert ^{2}R^{-1}-C\right) \leq \sup_{D\left({\rho}/{2}%
\right) }G & \leq  & G\left( q\right)  \\
&\leq &\frac{36}{\rho^{2}}\sup_{D\left( R\right) }\left( \left\vert \nabla 
\mathrm{Rm}\right\vert ^{2}R^{-1}\right) \leq \frac{C}{\rho},
\end{eqnarray*}%
Therefore, $\sup_{D\left(\rho/{2}\right) }\left( \left\vert \nabla \mathrm{Rm}%
\right\vert ^{2}S^{-1}\right) \leq C$ \ also when $G\left( q\right) >0$. In conclusion,  
\begin{equation*}
\sup_{D\left(\rho/{2}\right) }\left( \left\vert \nabla \mathrm{Rm}%
\right\vert ^{2}S^{-1}\right) \leq C.
\end{equation*}%
This completes Step 5.

\medskip
{\bf Step 6} ({\it Gradient estimate for $R$}). There exists $C>0$ such that 
\begin{equation}
|\na \ln R|^2 \leq C \ln (f+2) \quad \text{on} \ M .
\end{equation}

This is achieved by  adopting an argument in \cite{M-MTW}. The proof also depends on estimates (3.21) and (3.23), and the scalar curvature lower bound $R\ge C/f$ by Chow-Lu-Yang  \cite {CLY}. We refer the interested readers to the original paper of Munteanu-Wang \cite{MW} for details. 

\medskip
{\bf Step 7} ({\it Sharp $|Rc|$ estimate}). There exists $C>0$ such that 
\begin{equation}
|Rc| \leq C R \quad \text{on} \ M . 
\end{equation}

By direct computations and using the estimate  $|Rm|\le C$ in Step 2, the function $u:=\left\vert \mathrm{Ric}\right\vert ^{2}R^{-2}$
satisfies the differential inequality 
\begin{equation}
\ \ \Delta_{f-2\ln R} \left(u\right) \geq 3\,u^{2}\,R-c\,u\,R \label{B1}
\end{equation}%
on $M\backslash D\left(r_{0}\right) $. Here, $c>0$ is a universal constant. 

Now, by taking the same cut-off function $\psi \left( f\right) $ as in  Step 5 and applying the gradient estimate (3.28), we have
\begin{eqnarray*}
\Delta _{{f-2\ln R}} \left(\psi\right)  &=&\Delta _{f}\psi +2\left\langle \nabla \ln R,\nabla \psi
\right\rangle \\
&\geq &\frac{1}{\rho}\left( f-2\right) -\frac{2}{\rho}\left\vert \nabla \ln
R\right\vert \left\vert \nabla f\right\vert  \\
& \geq & \frac{1}{\rho}\left( f-2\right) -\frac{C}{\rho}\sqrt{f}\, \ln \left(f+2\right) . 
\end{eqnarray*}%
This shows that there exists a constant $t_{0}>0$ so that on $D\left(
\rho\right) \backslash D\left( t_{0}\right) ,$%
\begin{equation}
\Delta _{{f-2\ln R}} \left(\psi\right) \ge f/2\rho \geq 0.  \label{B2}
\end{equation}%
Using (\ref{B1}) and (\ref{B2}), for the function $G:=\psi ^{2}u$ we have
that on $M\backslash D\left( t_{0}\right) ,$%
\begin{eqnarray*}
\psi^{2}\Delta_{{f-2\ln R}} \left(G\right) & \geq & 3G^{2} R-cGR+G\Delta _{{f-2\ln R}} \left(\psi^{2}\right)+2 \psi ^{2}
\nabla u\cdot\nabla \psi ^{2} \label{B3} \\
& \geq & 3G^{2}R-cGR+2\nabla G \cdot \nabla \psi ^{2}
-8\left\vert \nabla \psi \right\vert ^{2}G.  \notag
\end{eqnarray*}%
By (\ref{A5'}) and the scalar curvature lower bound $Rf\geq C>0$, we have 
\[ \left\vert \nabla \psi \right\vert ^{2}G \leq \frac{1}{\rho}G \leq \frac{1}{C}RG. \]
Therefore, 
\begin{equation*}
\psi ^{2}\Delta _{{f-2\ln R}} \left(G\right) \geq \left( 3G^{2}-cG\right) R + 2\nabla
G \cdot \nabla \psi ^{2} .
\end{equation*}%
Now the maximum principle implies that $G$ must be bounded on $M\backslash D\left( t_{0}\right)$, 
and it follows that $|Rc|\le CR$.

\medskip
{\bf Step 8} ({\it Sharp $|Rm|$ estimate}). There exists $C>0$ such that 
\begin{equation}
|Rm| \leq C R   \qquad \text{on} \ M .
\end{equation}

This follows immediately from Lemma 2.5, the estimates $|\na Rm|^2 \le CR$ in Step 5 and $|Rc|\le CR$  in Step 7, and the scalar curvature lower bound $R\geq C/{f}$.

\medskip
{\bf Step 9} ({\it Further improved $|\na Rm|$ estimate}). There exists $C>0$ such that 
\begin{equation}
|\na Rm| \leq C R \quad \text{on} \ M .
\end{equation}
In particular, 
\begin{equation}
|\na \ln R| \leq C \quad \text{on} \ M .
\end{equation}

By (\ref{id}) and direct computations, the function $w:=\left\vert \nabla \mathrm{Rm}\right\vert ^{2}R^{-2}$ satisfies
the differential inequality 
\begin{eqnarray}
\Delta _{{f-2\ln R}} \left( w \right)&\geq &w-c\left\vert \mathrm{Rm}\right\vert w  \label{u3} \\
&\geq &w\left( 1-C_1R\right) .  \notag
\end{eqnarray}
Note, like (3.24), we again have a term linear in $w$ on the RHS of (3.35). 
By using the same cut-off function $\psi \left( f\right) $ with support in $D(\rho)$ as before, the function $G:=\psi ^{2}w$ satisfies
\begin{equation*}
\psi ^{2}\Delta _{{f-2\ln R}} \left( G \right) \geq \psi ^{2} G\left( 1-C_1R\right) +2 \psi \Delta _{{{f-2\ln R}}}\left(\psi
\right) G-6\left\vert \nabla \psi \right\vert ^{2}G+2\nabla G\cdot\nabla \psi ^{2} .  \label{u4}
\end{equation*}%
Let $q\in D\left( \rho\right) $ be the maximum point of $G$. If $q\in D\left(
r_{0}\right),$ then $w$ is bounded on $D\left( 
{\rho}/{2}\right) .$ So we only need to consider when $q\in
D\left( \rho\right) \backslash D\left( r_{0}\right)$. 
Furthermore, if $1-C_1R(q)\le 0$ so that $R^{-1}(q_0)\le C_1$, then $G(q)\le C_1^2|\na Rm|^2(q) \le C$, so  $G\leq G(q) \leq
C$ on $D\left( \rho\right) $. Again, this proves that $w$ is bounded on $%
D\left(\rho/{2}\right) $. So we may assume that $%
1-C_1R\left( q\right) \geq 0$. Now the maximum principle implies that at $q$   %
\begin{equation}
0\geq \psi ^{}\Delta _{f-2\ln R}\left( \psi \right) -3\left\vert \nabla
\psi \right\vert ^{2}.  \label{u5}
\end{equation}
Since $q\in D\left( \rho\right) \backslash D\left( r_{0}\right) $, (3.31) holds. Therefore, (%
\ref{u5}) implies that at $q,$ 
\begin{equation*}
\frac{f}{\rho}\psi \leq 6\left\vert \nabla \psi \right\vert ^{2}\leq 6\frac{f}{%
\rho^{2}} \quad \text{or}\quad \psi\le 6/\rho .
\end{equation*}%
This means that $f\left( q\right) \geq \rho-6$ and
\begin{eqnarray*}
\sup_{D\left( {\rho}/{2}\right) }\left( \left\vert \nabla 
\mathrm{Rm}\right\vert ^{2}R^{-2}\right) 
&\leq & 4 G\left( q\right) \\
&\leq &\frac{144}{\rho^{2}}\sup_{D\left( \rho\right) }\left( \left\vert \nabla 
\mathrm{Rm}\right\vert ^{2}R^{-2}\right) 
\leq \frac{c}{\rho},
\end{eqnarray*}%
where in the above last line we have used the estimate $|\na Rm|^2\le CR$ in Step 5  and the scalar curvature lower bound $Rf\geq
C>0$ by Chow-Lu-Yang \cite{CLY}. This again proves that $\left\vert \nabla \mathrm{Rm}\right\vert
^{2}R^{-2}$ is bounded on $D(\rho/2)$. In conclusion, %
\begin{equation*}
\sup_{D\left( {\rho}/{2}\right) }\left( \left\vert \nabla \mathrm{Rm}%
\right\vert^2 R^{-2}\right) \leq C.
\end{equation*}%
Since $\rho$ is arbitrary, this proves the estimate (3.33). 
\hfill \qedsymbol{}

\medskip

The fact that the curvature operator $Rm$ of 4-dimensional gradient shrinking Ricci solitons enjoys similar control (by the scalar curvature $R$) as in the dimension three case provides the hope for a possible classification of  (noncompact) Ricci shrinkers. As an application of Theorem 3.1, Munteanu and Wang  \cite{MW} proved sharp decay estimates for the Riemann curvature tensor and its covariant derivatives under the assumption that
the scalar curvature $R$ goes to zero at infinity. Indeed, they showed that if scalar curvature $R$ goes to zero at infinity, then it must decay at least quadratically as well: $R\le c/f$ for some constant $c>0$. By Theorem 3.1, this scalar curvature upper bound, together with the lower bound $R\ge C/f$ by Chow-Lu-Yang \cite{CLY}, implies that the Riemann curvature tensor $Rm$ must decay quadratically from above and below.  
This in particular enabled them  to conclude that such a Ricci shrinker must in fact be smoothly asymptotic to a cone at infinity.

Recall that an $n$-dimensional cone over a closed $(n-1)$-dimensional Riemannian manifold $(\Sigma^{n-1}, g_{\Sigma })$  is given by $[0,\infty )\times \Sigma $ endowed with Riemannian metric $%
g_{c}=dr^{2}+r^{2}\,g_{\Sigma }$. Denote by $E_{\rho}=(\rho,\infty )\times
\Sigma $ for $\rho\geq 0$ and define the dilation by $\lambda $ to be the map $%
\sigma _{\lambda }:E_{0}\rightarrow E_{0}$ given by $\sigma _{\lambda }(r,\theta)=(\lambda \,r,\theta ).$
Then a Riemannian manifold $(M^n, g)$ is said to be $C^{k}$ asymptotic to a cone $%
(E_{0},g_{c})$ if, for some $\rho>0,$ there is a diffeomorphism $\Phi: E_{\rho}\rightarrow M\setminus \Omega $ such that $\lambda ^{-2}\,\sigma_{\lambda }^{\ast }\,\Phi ^{\ast }\,g\rightarrow g_{c}$ as $\lambda\rightarrow \infty $ in $C_{loc}^{k}(E_{0},g_{c}),$ where $\Omega $ is a compact subset of $M.$

Now we can state precisely their conical structure result. 

\begin{theorem} {\bf (Munteanu-Wang \cite{MW})}  Let $\left( M^4, g, f\right) $ be a 4-dimensional complete noncompact gradient shrinking  Ricci soliton with scalar curvature
going to zero at infinity. Then there exists a cone $E_0$ such that $%
(M^4, g)$ is $C^k$ asymptotic to $E_0$ for all $k\ge 1.$
\end{theorem}

Kotschwar and L. Wang \cite{KW15} proved that if two complete noncompact gradient shrinking  Ricci solitons are $C^2$ asymptotic to the same cone then they must be isometric. Together with Theorem 3.2, this implies that the classification problem for 4-dimensional complete noncompact gradient shrinking Ricci solitons with scalar curvature going to zero at infinity is reduced to the one for limiting cones.

\section{Curvature estimates for 4D gradient expanding Ricci solitons}

Gradient expanding Ricci solitons may arise as Type III singularity models in the Ricci flow (under suitable positive curvature assumptions)  \cite{Cao97, ChenZhu}, and over which the matrix Li-Yau-Hamilton inequality (also known as the matrix  differential Harnack inequality) becomes equality \cite{Ha93, Cao92}. 
Indeed, Schulze-Simon \cite {SS13}  obtained gradient expanding solitons out of the asymptotic cones at infinity of solutions to the Ricci flow on complete noncompact Riemannian manifolds with bounded and non-negative curvature operator and positive asymptotic volume ratio; see also more recent work \cite{Der16, CD20}. 

By scaling, we may assume $\lambda=-1/2$, so that the gradient expanding Ricci soliton equation is of the form 
$R_{ij}+\nabla_i\nabla_jf=-\frac 1 2  g_{ij},$  or equivalently, 
\begin{equation}
 R_{ij} +\frac 1 2 g_{ij}=\nabla_i\nabla_j F, \quad \text{with} \quad F = -f+\frac n 2.
\end{equation} 
It is well-known that compact expanding solitons are necessarily Einstein.  Hence the study of expanding solitons has been focused on complete noncompact ones.  Here we are concerned with $4$-dimensional complete noncompact gradient expanding solitons with nonnegative Ricci curvature $Rc\ge 0$.

The simplest example of a complete (noncompact) gradient expanding  soliton is the {\it Gaussian expander }on $\RR^n$, with the standard flat metric and potential function $f(x)=-|x|^2/4$. In addition, Bryant \cite{Bryant} and the author \cite{Cao97} constructed non-flat rotationally symmetric expanding gradient Ricci and K\"ahler-Ricci solitons, including one-parameter families of expanding solitons with positive sectional curvature that are asymptotic to a cone at infinity, on $\RR^n$ and $\mathbb {C}^n$ respectively. Moreover,  the constructions in \cite{Cao94, Cao97} have been extended by Feldman-Ilmanen-Knopf \cite{FIK} to a construction of gradient  expanding K\"ahler-Ricci solitons on the complex line bundles $O(-k)$ ($k>n$) over the complex projective space ${\mathbb C}P^n$ ($n\geq 1$), and further generalized by Dancer-Wang  \cite{DW11}. 
For additional examples and other constructions, see, e.g., \cite{PTV00, Lauret01, GK04, BL07, DW09, FW11, BDGW15, AK19, Win19b} and the references therein. 
 
For curvature estimates, P.-Y. Chan \cite{Chan2} recently proved that if $(M^4, g, f)$ is a complete noncompact gradient expanding Ricci soliton with bounded scalar curvature $|R|\le R_0$ and proper potential function $f$ (i.e., $\lim_{r(x)\to \infty} F(x) =\infty$), then the curvature tensor $Rm$ is bounded. Note that $Rc\ge 0$ implies $0\le R\le R_0$, for some $R_0>0$, and $\lim_{r(x)\to \infty} F(x) =\infty$ (see Propositions 4.1-4.2 below).  It follows that $4$-dimensional complete gradient expanding solitons with nonnegative Ricci curvature $Rc\geq 0$  must have bounded Riemann curvature tensor $|Rm|\le C$.  
Motivated by the curvature estimates of Munteanu-Wang \cite{MW} for 4-dimensional gradient shrinking Ricci solitons as presented in Section 3, as well as the corresponding curvature estimates of Cao-Cui \cite{CaoCui} and Chan 
\cite{Chan1} for 4-dimensional gradient steady solitons (see also Section 5), it is natural to ask if one could also control the Riemann curvature tensor $Rm$ in terms of the scalar curvature $R$ for $4$-dimensional complete gradient expanding solitons with nonnegative Ricci curvature.

We remark that, as we have seen from the proof sketch of Theorem 3.1 in Section 3,  one of the important facts used repeatedly in the work of Munteanu and Wang \cite{MW} is the scalar curvature lower bound $R\ge C/ f$ for $n$-dimensional non-flat complete noncompact gradient shrinking soliton due to Chow-Lu-Yang \cite{CLY}.  Note that by the optimal asymptotic growth estimate of Cao-Zhou \cite{CaoZhou} on the potential function $f$ (Proposition 3.1), this is equivalent to the scalar curvature $R$ has at most quadratic decay at infinity. Another key feature, which resulted in the less standard maximum principle arguments in  Step 5 and Step 9 in the proof of Theorem 3.1, is  the differential inequality of the form $\Delta_{f} w\ge w$,  or  $\Delta_{f-2\ln R} (w) \ge w$, that is due to the curvature differential inequalities 
\[\Delta_{f} |Rm| \ge  |Rm|-c|Rm|^2  \quad \text{and} \quad \Delta_{f} |\na Rm| \ge  \frac 32 |\na Rm|-c|Rm||\na Rm| \]
satisfied by gradient shrinking solitons. In contrast, for gradient expanding solitons, we have instead
\[\Delta_{f} |Rm| \ge  -|Rm|-c|Rm|^2  \quad \text{and} \quad \Delta_{f} |\na Rm| \ge  -\frac 32 |\na Rm|-c|Rm||\na Rm|, \]
so the corresponding differential inequality we have to deal with is $\Delta_{f} w\ge -w$, or $\Delta_{f-2\ln R} (w) \ge -w$,  if we try to adopt a similar argument to get estimates for $w$. Unfortunately, this does not seem to work in the  expanding case. 

Despite the above mentioned differences with gradient shrinking solitons,  as it turns out, it is possible to obtain curvature estimates on $Rm $ and $\na Rm$ in terms of the scalar curvature $R$ for 4-dimensional gradient expanding solitons with nonnegative Ricci curvature, as shown in our very recent work with T. Liu \cite{CaoLiu}.

\begin{theorem} {\bf (Cao-Liu \cite{CaoLiu})}   Let $(M^4, g, f)$ be a $4$-dimensional complete noncompact 
gradient expanding Ricci soliton with nonnegative Ricci curvature $Rc\ge 0$. Then, there exists a constant $C>0$  such that, for any $0\leq a<1$, the following estimates hold: 
\begin{equation}
 |Rm|  \le \frac {C} {1-a} R^a \quad {\mbox{and}} \quad |\nabla Rm|\le \frac {C} {1-a} R^a \quad on \ M^4.  
\end{equation}
Moreover,  if in addition the scalar curvature $R$ has at most polynomial decay then 
\begin{equation}
 {|Rm|} \le C R  \quad on \ M^4. 
\end{equation}
\end{theorem}

By using Theorem 4.1, we also obtain a very useful result about asymptotic curvature behavior for $4$-dimensional gradient expanding Ricci solitons with $Rc\geq 0$. 

\begin{theorem} {\bf (Cao-Liu \cite{CaoLiu})}  Let $(M^4, g, f)$ be a  $4$-dimensional complete noncompact gradient expanding Ricci soliton with nonnegative Ricci curvature $Rc\ge 0$. Assume it has finite asymptotic scalar curvature ratio 
\begin{equation}
 \limsup_{r\to \infty} R  r^2< \infty, 
\end{equation}
where $r=r(x)$ is the distance function to a fixed base point in $M$. 
Then $(M^4, g, f)$ has finite {\em asymptotic curvature ratio }
\begin{equation}
A := \limsup_{r\to \infty} |Rm| r^2< \infty.  
\end{equation}
\end{theorem} 

\begin{remark} If we only assume $Rc\geq 0$ outside some compact set $K\subset M$, then the conclusions in Theorem 4.1 and Theorem 4.2 remain true over  $M\setminus K$. 
\end{remark}

As an application of Theorem 4.2, by combining with a result of Chen and Deruelle \cite{ChenDer} (see Theorem 1.2 of  \cite{ChenDer}), we get 

\begin{corollary} {\bf (Cao-Liu \cite{CaoLiu})}  Let $(M^4, g, f)$ be a $4$-dimensional complete noncompact non-flat 
gradient expanding Ricci soliton with nonnegative Ricci curvature $Rc\ge 0$ and finite asymptotic scalar curvature ratio.  
Then $(M^4, g, f)$ has a $C^{1, \alpha}$ asymptotic cone structure at infinity, for any $\alpha \in (0, 1)$. 
\end{corollary}

\subsection {Background material for expanding solitons with $Rc\ge 0$} 

According to Munteanu-Wang \cite{MW12},  the potential function of a general complete gradient expanding Ricci soliton $(M^n, g, f)$ satisfying (4.1),  with respect to any reference point $p\in M$, satisfies the following estimates
\begin{equation}
\frac 1 4r^2-Cr^{3/2} \sqrt{\ln r} \le \sup_{\partial B_p(r)} F(x) \le \frac 1 4 r^2+Cr, 
\end{equation}
for some constant $C>0$. 

On the other hand, if $(M^n, g, f)$ has nonnegative Ricci curvature $Rc\ge 0$ then we have the following well-known fact about the asymptotic behavior of the potential function (see, e.g., Lemma 5.5 in \cite{Cao et al} or Lemma 2.2 in \cite{ChenDer}).

\begin{proposition}  Let $(M^n, g , f)$ be a complete noncompact gradient expanding Ricci soliton satisfying (4.1) and with nonnegative Ricci curvature $Rc \ge 0$. Then, there exists some constant $c_1 >0$ such that, outside some compact subset of $M^n$, the potential function $F=-f+n/2$ satisfies the estimates
\begin{equation}
\frac 1 4(r(x)-c_1)^2\le F(x)\le \frac 1 4 (r(x)+2\sqrt{F(x_0)})^2,  
\end{equation}
where $r(x)$ is the distance function from some fixed base point in $M^n$. In particular, $F$ is a strictly convex exhaustion function achieving its minimum at its unique interior point $x_0$, which we shall take as the base point, and the underlying manifold $M^n$ is diffeomorphic to ${\mathbb R}^n$.
 \end{proposition}

Another useful fact is the boundedness of the scalar curvature of a gradient expanding soliton with nonnegative Ricci curvature; see Ma-Chen \cite{MC} (and also \cite{DZ, Der17}). 

\begin{proposition} 
Let $(M^n, g, f)$ be a complete noncompact gradient expanding Ricci soliton with nonnegative Ricci curvature $Rc \ge 0$. Then its scalar curvature $R$ is bounded from above, i.e., 
\[R\le R_0 \quad \mbox{for some positive constant} \ R_0.\]
Moreover, $R>0$ everywhere unless  $(M^n, g, f)$ is the Gaussian expanding soliton. 
\end{proposition}

\begin{remark}
In \cite{MC}, the authors asserted that either $R>0$ or $(M^n, g)$ is Ricci flat.  However, a complete non-compact Ricci-flat gradient expanding Ricci soliton must be the Gaussian expanding soliton on ${\RR}^n$ (see, e.g., Proposition 3.1 in \cite{PW}). 
\end{remark}

Furthermore, by combining  (2.1) and (2.4), we obtain 
$$\Delta f -|\nabla f|^2=f-\frac n 2, \quad  {\mbox {i.e.}}, \quad   \Delta_f f=f-\frac n 2$$
where $\Delta_f =:\Delta -\nabla f\cdot \nabla$ is the weighted Laplace operator. 
Equivalently,  for $F=-f+n/2$, 
\begin{equation}
 |\na F|^2=F-R-\frac n 2 \qquad \mbox{and}  \qquad \Delta_f F =F. 
\end{equation}

Again, by Lemmas 2.2-2.4, we have the following differential identities and inequalities for any complete gradient expanding Ricci soliton satisfying Eq. (4.1).
\begin{eqnarray}
\Delta_{f} R &=&-R-2|Rc|^2, \notag \\
\Delta_{f} R_{ik} &=&-R_{ik} -2R_{ijkl}R_{jl},\notag \\
\Delta_{f} {Rm} &=&  -Rm+ Rm\ast Rm,\notag \\
\na_lR_{ijkl} &=& \na_jR_{ik}-\na_i R_{jk}=R_{ijkl}\na_lf, \notag\\
\Delta_{f} |Rc|^2 & \ge & 2|\na Rc|^2-2|Rc|^2-4|Rm| |Rc|^2, \label{id2}\\
\Delta_{f}|Rm|^2  &\ge & 2|\na Rm|^2 - 2|Rm|^2-c|Rm|^3, \notag\\
\Delta_{f} |Rm| &\ge & -|Rm|-c|Rm|^2, \notag \\
\Delta_{f} |\na Rm|^2 &\ge & 2|\na^2 Rm|^2  -3 |\na Rm|^2-c|Rm| |\na Rm|^2,  \notag\\
\Delta_{f} |\na Rm| &\ge & -\frac 32 |\na Rm|-c|Rm||\na Rm| . \notag
\end{eqnarray}
Here, $c>0$ is some universal constant depending only on the dimension $n$.

\subsection {Sketch of the proofs of Theorem 4.1 and Theorem 4.2} \ In this subsection, we shall provide a sketch of the proofs of Theorem 4.1 and Theorem 4.2, and refer the readers to our paper \cite{CaoLiu} for further details.  

We denote by
\begin{equation} 
D(t):= \{x\in M : F(x) \leq t \}. 
\end{equation}
As a consequence of Lemma 2.5, Proposition 4.1 and Proposition 4.2, one easily gets the following

\begin{lemma} 
Let $(M^4, g, f)$ be a $4$-dimensional complete noncompact gradient expanding Ricci soliton satisfying Eq. (4.1) and with nonnegative Ricci curvature.
Then, for any constant $\Lambda >1$, there exists some constant $r_0>0$ (depending on $\Lambda$) such that 
\[|\na Rc|^2 \ge \frac {\Lambda} {2A_0^2} |Rm|^2 - \Lambda |Rc|^2 \quad \mbox{outside}\  D(r_0).\] 
\end{lemma}

Let us first consider the proof of Theorem 4.1. 

\medskip
\noindent {\it Sketch of the Proof of Theorem 4.1}.  \ Again the proof consists of several estimates on $Rm$ and $|\na Rm|$. 

\medskip
{\bf Step A} ({\it Initial $|Rm|$ estimate}).  There exists a constant $C>0$, depending on the constant $A_0$ (in Lemma 2.5) and the scalar curvature upper bound $R_0$ (in Proposition 4.2), such that
\begin{equation}
|Rm|  \le C \quad \text{on} \ M .
\end{equation} 

\begin {remark}  As we pointed out before, The above estimate $|Rm|\le C$ in Step A is a special case of a more general result due to P.-Y. Chan \cite{Chan2}. 

\end {remark}

 The proof follows essentially from the same argument as in \cite{MW} (Step 2 in the proof of Theorem 3.1) or the proof of Theorem 3.1 in \cite{CaoCui}. Indeed, 
by the assumption of $Rc\ge0$ and Proposition 4.2, it follows that
\begin{equation}
|Rc|\le R\le R_0. 
\end{equation}
Thus, by (4.12) and Lemma 4.1 (taking $\Lambda=4A_0^2$), we get
\begin{equation}
|\nabla Rc|^2\geq 2|Rm|^2 -4A_0^2R_0^2 
\end{equation}
outside some compact set $D(r_0)$. 

On the other hand, by (\ref{id2}), (4.12) and (4.13), 
\begin{eqnarray*}
\Delta_f(|Rm|+\lambda |Rc|^2) & \geq & 2\lambda |\nabla Rc|^2-c |Rm|^2 -\lambda C_1 |Rm| - \lambda C_1\\
& \geq & (4\lambda-c) |Rm|^2 -\lambda C_1 |Rm| - \lambda C_2,
\end{eqnarray*}
where $C_1=C_1(R_0)>0$ depends on the scalar curvature upper bound $R_0$, and $C_2=C_2(A_0, R_0)$ depends on the constants $A_0$ and $R_0$.
By picking 
$$4\lambda=c+2,$$ it follows that 
\begin{eqnarray*}
\Delta_f(|Rm|+\lambda |Rc|^2) & \ge & 2|Rm|^2-\lambda C_1 |Rm|-\lambda C_2\\
 & \ge & (|Rm|+\lambda |Rc|^2)^2 -C_1 (|Rm|+\lambda |Rc|^2)-C_2. 
\end{eqnarray*}
Let $u = |Rm|+\lambda |Rc|^2$, then 
\begin{equation*}
\Delta_f u \ge u^2-C_1 u-C_2 .
\end{equation*}

Now it follows from a standard cut-off function and maximum principle argument (as in Step 1 of the proof of Theorem 3.1) that  
$$ \sup_{x\in M}|Rm|\le \sup_{x\in M} \left(|Rm|+\lambda |Rc|^2\right)\le C.$$
\hfill \qedsymbol{}

\noindent
{\bf Step B} ({\it The $|Rm|$ estimate}). \ There exists some constant $C>0$, depending on $A_0$ and $R_0$, such that for any $0<a<1$, 
\begin{equation}
 |Rm|  \le \frac C {1-a} R^{a} \quad \mbox{on}\ M^4 .
 \end{equation}

By Proposition 4.2, it suffices to consider the case when the scalar curvature $R>0$ everywhere on $M$.  
For any $a\in (0, 1)$, by direct computations and using (2.5), we have 
\begin{eqnarray}
\Delta_{f-2a\ln R} \large(R^{-a}{|Rm|}\large) & \ge &  -C R^{-a}|Rm|+a(1-a)R^{-a}|Rm| |\na \ln R|^2, \notag \\
 \Delta_{f-2a\ln R} \large (R^{-2a}{|Rc|^2}\large) &  \ge & 2(1-a)R^{-2a} {|\na Rc|^2}-CR^{-2a} {|Rc|^2} .
\end{eqnarray} 
Here, $C>0$ is a constant depending on the constants $A_0$ and $R_0$.

Now, we consider the function 
\[v:=\frac {|Rm|}{R^a} +\frac {|Rc|^2}{R^{2a}}.\]
Then, by (4.15) and Lemma 4.1 (with $\Lambda =4A_0^2$), 
\begin{eqnarray*} 
\Delta_{f-2a\ln R} (v) & \ge  &2(1-a)  \frac {|\na Rc|^2} {R^{2a}}  -C_1 \frac {|Rm|}{R^a}-C_2 \frac {|Rc|^2} {R^{2a}}\\
& \ge & 2(1-a) v^2 - C_1 v-C_2 \qquad  \mbox{on} \ M\setminus D(r_0), 
 \end{eqnarray*} 
where  $C_1>0$ and $C_2>0$ depend on $A_0$ and $R_0$.

Next, let $\varphi (t)$ be a smooth  function on $\mathbb R^{+}$  defined by
$$ {\varphi (t) =\left\{
       \begin{array}{ll}
  1, \ \ \quad  \rho\le t\le 2\rho ,\\[4mm]
    0, \ \ \quad  0\le t\le \rho/2 \ \ \mbox{or} \ \ t\ge 3\rho,
       \end{array}
    \right.}$$
such that 
\begin{equation}
 t^2 \left(|\varphi'(t)|^2+|\varphi''(t)|\right)\le c 
\end{equation}
for some universal constant $c$ and arbitrary large $\rho>2r_0$. Take $\varphi=\varphi (F(x))$ as a cut-off function, with support in $D(3\rho)$. Note that 
\begin{equation}
 |\nabla \varphi|=|\varphi'\nabla F|\le  \frac {c} {\rho} |\na F| \qquad  \mbox{and} 
\end{equation}
\begin{equation}
|\Delta_f \varphi| =|\varphi' \Delta_f F +\varphi''|\nabla F|^2|\le \frac {c} {\rho}F +  \frac {c} {\rho^2}|\na F|^2\le 2c
\end{equation}
on $D(3\rho)\setminus D(\rho/2)$.

Set $G=\varphi^2 v$. Then,  
\begin{eqnarray*} \varphi^2 \Delta_ {f-2a\ln R} (G)
&\ge & 2 (1-a) G^2  -C_1 G -C_2  +G\Delta_{f-2a\ln R}  (\varphi^2) \\
& & +2\nabla G\cdot\nabla \varphi^2 -8G|\na \varphi|^2. 
\end{eqnarray*}
Note that, by  (4.17), 
\[ |\na \varphi|^2\leq \frac {c^2}{\rho^2} |\na F|^2\leq \frac c {\rho}.\]
Moreover, 
\[ \Delta_{f-2a\ln R}  (\varphi) = \Delta_{f}  (\varphi) +2a \na \varphi \cdot \na \ln R \]
and 
\begin{equation}
 |\na \ln R| \leq 2\frac {|Rc|} {R} |\na F| \leq 2|\na F|. 
\end{equation}
Thus, by (4.17)-(4.19), we obtain 
\[\Delta_{f-2a\ln R}  (\varphi) \ge -6c \] 
for some universal constant $c>0$.  Therefore,  
\[\varphi^2 \Delta_ {f-2a\ln R} G \ge 2 (1-a) G^2  -C G -C +2\nabla G\cdot\nabla \varphi^2.\]
Now, by the standard maximum principle argument, it follows that $v\le \frac {C}{1-a}$ on $M^n$ for some constant $C>0$ depending on $A_0$ and $R_0$. Therefore, 
\[ \frac {|Rm|}{R^a} \le \frac {|Rm|}{R^a} +\frac {|Rc|^2}{R^{2a}}\le \frac {C}{1-a} \quad \mbox{on} \ M^4.\]
\hfill \qedsymbol{}

We can also prove a similar estimate for the covariant derivative $\na Rm$ of the curvature tensor. 

\medskip
\noindent
{\bf Step C} ({\it The $|\na Rm|$ estimate}). \ There exists some constant $C>0$, depending on $A_0$ and $R_0$,  such that for any $0\leq a<1$, 
\begin{equation}
 |\na Rm|  \le \frac C {1-a} R^{a} \quad \mbox{on}\ M^4 . 
\end{equation}

By (\ref{id2}) and the estimate  $|Rm|\le C$ in Step A, we have 
\[\Delta_{f} |\na Rm|\ge -\frac 3 2 |\na Rm|-c|Rm||\na Rm|\ge -C |\na Rm|,\]
\[\Delta_{f} |Rm|^2\ge 2|\na Rm|^2-2|Rm|^2 -C|Rm|^3\ge 2|\na Rm|^2-C|Rm|^2.\]
Then, by direct computations, we obtain
\begin{eqnarray*} 
\Delta_{f}\frac {|\na Rm|}{R^a}  & \ge &  -C \frac {|\na Rm|} {R^{a}} + a(1-a)\frac {|\na Rm|} {R^{a}} |\na \ln R|^2 -2a \na \left(\frac {|\na Rm|} {R^{a}}\right)\cdot \na \ln R \\
\Delta_{f} \left(\frac {|Rm|^2} {R^{2a}} \right) & \ge  & 2(1-a)  \frac {|\na Rm|^2} {R^{2a}}  -C \frac {|Rm|^2} {R^{2a}} -2a\na \left(\frac {|Rm|^2} {R^{2a}}\right)\cdot \frac {\na R}{R}.
\end{eqnarray*} 

Thus, the function
\[ w:=\frac {|\na Rm|} {R^a} + \frac {|Rm|^2} {R^{2a}}\] satisfies the differential inequality 
\[ \Delta_{f-2a\ln R} (w) \ge  (1-a) w^2 - C_1 w-C_2. \]
Here, we have used the estimate on $|Rm|$ in Step B. 

Now the desired estimate (4.20) follows from a similar maximum principle argument as in the proof of Step B above. 
\hfill \qedsymbol{}

 \medskip
\noindent
{\bf Step D} ({\it The shap estimate on $|Rm|$ when $R$ has polynomial decay}). \  Suppose  in addition the scalar curvature $R$ has at most polynomial decay, then 
there exists a constant $C>0$ such that 
$$ {|Rm|} \le C R \quad \text{on} \ M .$$

Assume that the scalar curvature satisfies $R\ge C/{r^d(x)}$ for some $d\ge 1$ outside a compact set. Then,   we get 
\begin{equation} 
R \ge \frac {C} {r(x)^d}  \ge \frac {C} {|\na f|^d} . 
\end{equation}
It follows from Lemma 2.5, $|Rc|\le R$, (4.21) and (4.20) (with $a=1-1/d$) that 
\[ |Rm| \le A_0 \left( |Rc| + \frac {|\na Rc| }{|\na f|}\right) \le A_0 ( R +  Cd R^{1-1/d} R^{1/d})\le C_1 R.\]
This finishes the sketch proof of Theorem 4.1.

\hfill \qedsymbol{}

\medskip
\noindent {\it Proof of Theorem 4.2.} \ Recall that, by Lemma 2.5, we have 

\begin{equation} 
|Rm|\le A_0\ \left(|Rc|+ \frac {|\nabla Rc|} {|\nabla f|}\right) . 
\end{equation}

On the other hand,  it follows from $Rc\geq 0$ and the assumption of the finite asymptotic scalar curvature ratio (4.4) that 
\begin{equation} 
|Rc|\leq R \leq \frac {C_1} {r^2}, 
\end{equation}
for some constant $C_1>0$. 
Moreover, by picking $a=1/2$ in (4.20), we get  
\begin{equation} 
 |\na Rc|\le c|\na Rm| \leq C R^{1/2} \leq \frac {C_2}  {r},
\end{equation}
while 
\begin{equation} 
 |\na f|^2 =-f -R= O(r^2)
\end{equation} 
by (2.4), Proposition 4.1 and Proposition  4.2. 

Plugging (4.23)-(4.25) into (4.22) leads to 
\[|Rm| \leq \frac C {r^2}.\]
\hfill \qedsymbol{}

\subsection {Curvature estimates for expanders with $R>0$ and proper $f$} \! It turns out that one can also adapt the arguments in the proof of Theorem 4.1 to obtain curvature estimates for $4$-dimensional  
complete noncompact gradient expanding Ricci solitons with bounded, positive scalar curvature $0<R\le R_0$ and proper potential function.
  
\begin{theorem} {\bf (Cao-Liu \cite{CaoLiu})} Let $(M^4, g, f)$ be a $4$-dimensional complete noncompact gradient expanding Ricci soliton with bounded and positive scalar curvature  $0<R\leq R_0$. Assume that $f$ is proper so that $\lim_{r(x)\to \infty}f(x)=-\infty$. Then, for any $\alpha \in (0, 1/2)$, 
$$ |Rm|  \leq C_{\alpha} R^{\alpha} \quad {\mbox{and}} \quad |\nabla Rm|\leq C_{\alpha} R^{\alpha} \quad on \ M^4 $$
for some positive constant $C_{\alpha}>0$ with $C_{\alpha}\to \infty$ as $\alpha \to 1/2$.  
\end{theorem}

\begin{remark} For $\alpha=0$, the estimate $|Rm| \le C$ in Theorem 4.3 is a special case of Chan \cite{Chan2}, where he assumed $f$ is proper but $R$ is only bounded (i.e., $|R| \le R_0$). 

\end{remark}

\medskip
\noindent {\it Sketch of the Proof of Theorem 4.3}.  Again, we shall divide the proof of Theorem 4.3 into several estimates. 

\medskip
{\bf Step I}  ({\it Initial $|Rc|$ estimate}). There exists a constant $C>0$ such that 
\begin{equation}
|Rc|^2 \le C R^{1/2} \quad \text{on} \ M .
\end{equation}

This follows essentially from the same argument as in Munteanu-Wang \cite{MW} (Step 1 in the proof of Theorem 3.1).  By  (\ref{id2}) and Lemma \ref{Curv}, we have
\begin{equation*}
\Delta _{f}\left\vert \mathrm{Rc}\right\vert ^{2} 
\geq 2\left\vert \nabla \mathrm{Rc}\right\vert ^{2}-\frac{c}{|\na F|}%
\left\vert \nabla \mathrm{Rc}\right\vert \left\vert \mathrm{Rc}\right\vert
^{2}   -c\left\vert \mathrm{Rc%
}\right\vert ^{3}.  \notag
\end{equation*}%
For any $a \in (0, 1)$, set  $u:=|Rc|^2/R^{a}$. Then direct computations  and using (2.5)  yield the following differential inequality, 
\begin{equation*}
\Delta _{f}u\geq \left( 2a-\frac{c}{1-a}\frac{R}{|\na F|^2}\right)
u^{2}R^{a-1}-c\,u^{\frac{3}{2}}\,R^{\frac{a}{2}}-c\,u, 
\end{equation*}%
where $c>0$ is a universal constant. 
Now, setting $a=1/2$, then 
\begin{eqnarray*}
\Delta _{f}u & \geq & \left( 1-2cR_0/{|\na F|^2}\right)
u^{2}R_0^{-1/2}-c\,u^{\frac{3}{2}}R_0^{1/2}\,-C\,u \\
& \geq & \frac 1 2 u^{2}R_0^{-1/2}-C\,u^{\frac{3}{2}}\,-C\,u  \quad \text{on} \ M\backslash D\left( r_{1}\right), 
\end{eqnarray*}
where $u=R^{-1/2}|Rc|^2$, $C> 0$ depends only on $R_0$, and $r_1>0$ is chosen so that 
$$ 1-2cR_0/{|\na F|^2} \geq 1/2  \quad \text{on}\ M\backslash D\left( r_{1}\right),$$
which is possible by Propositions 4.1-4.2 and (4.8).

Now a standard cut-off function and maximum principle argument shows that $u\le C$ on $M$.

\medskip
{\bf Step II}  ({\it Initial $|Rm|$ estimate}). There exists a constant $C>0$ such that 
\begin{equation}
|Rm| \le C \quad \text{on} \ M .
\end{equation}

Since we now have the bound $|Rc|\le C$, this follows from essentially the same argument as in Step A of the proof of Theorem 4.1. 

\medskip
{\bf Step III}  ({\it Improved $|Rc|$  estimate}). There exists a constant $C>0$ such that 
\begin{equation}
|Rc|^2 \le C R \quad \text{on} \ M .
\end{equation}

For any $0<a \leq 1$, by  (2.5), (\ref{id2}), (4.27) and direct computations, we obtain
\begin{equation}
 \Delta_{f} \left(\frac {|Rc|^2} {R^a} \right) \ge  \frac {2(1-a)}{1+a} \frac {|\na Rc|^2} {R^a} -(4C_0+2)  \frac {|Rc|^2} {R^a} + 2a  \frac {|Rc|^4} {R^{1+a}} . 
\end{equation}
Thus, taking $a=1$ and setting $u=\frac {|Rc|^2} {R}$, we have
\[ \Delta_{f} (u) \geq 2u^2-(4C_0+2)u.\]
Now, estimate (4.28) follows from a standard cut-off function and maximum principle argument. 

\medskip
{\bf Step IV}  ({\it The  $|Rm|$ estimate}). There exists a constant $C>0$ such that  for any $\alpha \in (0, 1/2)$, 
\begin{equation}
|Rm|  \leq C_{\alpha} R^{\alpha} \quad \text{on} \ M .
\end{equation}

We  consider the quantity  of the following form suggested to us by P.-Y. Chan, 
\[\frac{|Rm|^{2p}}{R^b}+\frac{|Rc|^{2}}{R^a}.\] 
By (\ref{id2}), (2.5),  the estimate $|Rm|\le C$ in Step II and direct computations, for any $p>0$ and $b>0$ we get  (whenever $|Rm|\neq 0$ in case $p < 1$) 
\begin{equation} 
 \Delta_{f} \left(\frac {|Rm|^{2p}} {R^{b}} \right) \ge  \left[2p(2p-1)-\frac {4p^2b} {(b+1)}\right] \frac {|Rm|^{2p-2}} {R^{b}} |\na |Rm||^2 -Cp \frac {|Rm|^{2p}} {R^{b}}.
\end{equation}  

Now, for any $0<a<1$ and $b\in (0,1)$, let $2p=1+b \in (1, 2)$. Then, by (4.29), (4.31), the estimate $|Rc|^2\le CR$ in Step III,  and applying Lemma 4.1 (with $\Lambda=2A_0^2$), we obtain  
\begin{equation} 
 \Delta_{f} \left(\frac {|Rm|^{1+b}} {R^{b}} +\frac {|Rc|^2} {R^a} \right)
\geq  (1-a) \frac {|Rm|^2} {R^a} -C_1\left( \frac {|Rm|^{1+b}} {R^{b}} + \frac {|Rc|^2} {R^a}\right) -C_2 . 
\end{equation}  
Then, for  any $\alpha \in (0, 1/2)$, we choose $a=2\alpha \in (0, 1), \ b= \alpha/(1-\alpha)<a$ and set 
\[v=\frac {|Rm|^{1+b}} {R^{b}} +\frac {|Rc|^2} {R^a}.\]
Without loss of generality, we may assume $v\leq  2 R^{-b} |Rm|^{1+b} $. Then,
\begin{equation}
 \Delta_{f} (v) \geq \frac {1-2\alpha} {4} v^{2(1-\alpha)}-C_1 v-C_2 .
\end{equation}
Since $1< 2(1-\alpha)$, by a standard maximum principle argument we can conclude from (4.32) that $v\leq C_{\alpha}$ for some positive constant $C_{\alpha}>0$, with $C_{\alpha}\to \infty$ as $\alpha \to 1/2$. 
Therefore, $|Rm|^{1+b} \leq C_{\alpha} {R^{b}}$,  or equivalently
\[ |Rm|\leq C_{\alpha} R^{\frac b {1+b}}= C_{\alpha} R^{\alpha} \quad on \ M^4 .\]

\medskip
{\bf Step V}  ({\it The  $|\na Rm|$ estimate}). There exists a constant $C>0$ such that  for any $\alpha \in (0, 1/2)$, 
\begin{equation}
|\na Rm|  \leq C_{\alpha} R^{\alpha} .
\end{equation}

By direct computations, we have 
\[ \Delta_{f} \left(\frac {|\na Rm|^{1+b}} {R^{b}} +\frac {|Rm|^2} {R^a} \right) \geq (1-a) \frac {|\na Rm|^2} {R^a} -C_3\left( \frac {|\na Rm|^{1+b}} {R^{b}} + \frac {|Rm|^2} {R^a}\right), \]
which is similar to (4.31). 

Now estimate (4.32) follows  similarly as in Step IV. 

\hfill \qedsymbol{}

\section{Curvature estimates for 4D gradient steady Ricci solitons}

In this section, we present the curvature estimates for $4$-dimensional  gradient steady Ricci solitons proved by Cao-Cui \cite{CaoCui} and P.-Y. Chan \cite{Chan1}. In addition, we derive some new estimates for $4$-dimensional  gradient steady Ricci solitons (Theorem 5.2 and Theorem 5.3). 

Let $(M^4, g, f)$ be a $4$-dimensional  gradient steady Ricci solitons satisfying the equation $R_{ij}+\nabla_i\nabla_jf=0$. Equivalently, by setting $F=-f$, we have 
\begin{equation}
R_{ij}=\nabla_i\nabla_jF .
\end{equation}

As is well-known,  compact steady solitons must be Ricci-flat. In dimension $n=2$, Hamilton \cite{Ha88} discovered the first example of a complete noncompact gradient steady
soliton on $\mathbb R^2$, called the {\it cigar soliton}, where the metric is given by
$$ ds^2=\frac{dx^2 +dy^2}{1+x^2+y^2}.$$
The cigar soliton has potential function $F=\log (1+x^2+y^2)$ (which has linear growth in geodesic distance), positive (scalar) curvature $R=4e^{-F}$, and is asymptotic to a round cylinder at infinity.  Furthermore, Hamilton \cite{Ha88} showed that the only
complete steady soliton on a two-dimensional manifold with
bounded (scalar) curvature $R$ which assumes its maximum
at an origin is, up to scaling,  the cigar soliton. 
For $n\geq 3$, Bryant \cite{Bryant} proved that
there exists, up to scaling, a unique complete rotationally symmetric gradient Ricci
soliton on $\Bbb R^n$. The Bryant soliton has linear growth potential function $F$, positive curvature operator $Rm>0$, linear curvature decay $R\le c/F$, and volume growth of geodesic balls $B(0,r)$ on the order of $r^{(n+1)/2}$. In the K\"ahler case,
the author \cite{Cao94} constructed a complete $U(m)$-invariant gradient steady K\"ahler-Ricci soliton on $\mathbb{C}^m$, for $m\geq 2$, with positive sectional curvature and linear growth potential function $F$; its geodesic ball of radius $r$ has volume growth on the order of $r^{m}$ and the scalar curvature has linear decay at infinity. Note that in each of these three examples, the maximum of the scalar curvature $R$ is attained at the origin. One can find additional examples of steady solitons, e.g., in \cite {Iv, FIK, DW1, DW11, BDGW} etc; see also \cite{Cao08b} and the references therein.

Inspired by the  work of Munteanu-Wang \cite{MW}, Cui and the author \cite{CaoCui}\footnote{The preprint was posted on the arXiv in 2014.}  studied curvature estimates of 4-dimensional complete noncompact gradient steady solitons.

\begin{proposition} {\bf (Cao-Cui \cite{CaoCui})} 
Let $(M^4, g, f)$, which is not Ricci-flat,  be a  complete noncompact $4$-dimensional
gradient steady Ricci soliton. If $\lim_{x\to \infty} R(x)=0$, then, for each $0<a<1$, there exists
a constant $C>0$ such that
\begin{equation}
 |Rc|^2\le C {R^a} \qquad  \mbox{and} \qquad \sup_{x\in M} |Rm| \le C .
\end{equation}
Suppose in addition $R$ has at most polynomial decay. Then, for each $0<a<1$, there exists a constant $C>0$ such that
\begin{equation}
 |Rm|^2 \leq C R^{a} . 
\end{equation}
\end{proposition}

Subqequently,  using  the estimates in (5.2), P.-Y. Chan \cite{Chan1} improved the curvature estimate (5.3) and obtained the following sharp result  without assuming the polynomial decay of $R$.

\begin{theorem} {\bf (Chan \cite{Chan1})}  Let $(M^4, g, f)$ be a  $4$-dimensional  complete non-Ricci flat gradient steady Ricci soliton with $\lim_{x\to \infty} R(x)=0$. Then there exists a positive constant $C>0$ such that
\[  |Rm| \leq C R \qquad \mbox{on} \ M. \]
\end{theorem}

\begin{remark}
One of the key steps in Chan's proof is to obtain the sharp Ricci curvature estimate. Namely,  
there exist positive constants $r_0>0$, and $C_1>0$ depending on the constant $A_0$ in Lemma 2.5 and $\max |Rc|/R$ over the ball $B_{x_0}(r_0)$, such that
\begin{equation}
 |Rc|\leq C_1 R \qquad \mbox{on} \ M .
\end{equation}
This is achieved by applying the maximum principle argument to the function 
\[u:=|Rc|^2+|Rc|-C R\]
for a suitably large constant $C>0$ such that $u<0$ on $\partial B_{x_0}(r_0)$ and showing that  
\begin{equation}
\Delta_f (|Rc|^2+|Rc|)\geq (6A_0+20A_0^2) |Rc|^2 
\end{equation}
by using the curvature estimates in (5.2). We refer the reader to Chan \cite{Chan1} for more details. 
\end{remark}

Moreover, Cui and the author proved the following estimates for $4$-dimensional
gradient steady Ricci soliton with positive Ricci curvature $Rc>0$. 

\begin{proposition}  {\bf (Cao-Cui \cite{CaoCui})} 
Let $(M^4, g, f)$ be a complete noncompact $4$-dimensional
gradient steady Ricci soliton with positive Ricci curvature $Rc>0$ such that the scalar curvature $R$ attains its maximum
at some point $x_0\in M^4$. Then, $(M^4, g, f)$ has bounded Riemann curvature tensor, 
\begin{equation}
 |Rm|  \le C  \quad \text{on} \ M 
\end{equation}
for some constant $C>0$. 
Moreover, if  in addition $R$ has at most linear decay, then
\[ |Rm|  \le CR  \quad \text{on} \ M .\]
\end{proposition}

Similar to Theorem 5.1, we can also remove the assumption that $R$ has at most linear decay in Proposition 5.2. 

\begin{theorem}
Let $(M^4, g, f)$ be a complete noncompact $4$-dimensional
gradient steady Ricci soliton with positive Ricci curvature $Rc>0$ such that the scalar curvature $R$ attains its maximum at some point $x_0\in M$. Then, $$ |Rm|  \le C R \quad \mbox{on} \ M.$$
\end{theorem}

\begin{proof}  Since the scalar curvature $R>0$, by scaling the metric $g$, we can normalize Eq. (2.3) as
\[ R+|\na f|^2 =1. \]

Thus, 
\begin{equation}
 R\le 1    \quad \text{and} \quad |\na f|^2 \le 1 \quad \text{on}\  M .  
\end{equation}

At the same time, since $Rc>0$, $F=-f$ is convex. From \cite{CaoChen} we also know that $F$ grows linearly in geodesic distance,
\begin{equation}
 c_1 r(x)-c_2 \le F(x) \le r(x) +F(x_0)
\end{equation}
for some constants $c_1>0$ and $c_2>0$, where $r(x)=d(x_0, x)$ is the distance function from $x_0$.  Then it follows that 
$x_0$ is also the unique minimum/critical point of $F$, and $R(x_0)=\max_{x\in M} R=1$. 
Next, by (2.5), we have
\begin{equation}
\Delta_{f} (R^{-a})\geq a(a+1)R^{-a}|\na \ln R|^2 .
\end{equation}
Also, by Lemma 2.3, curvature estimate (5.6) and direct computations, 
 \begin{eqnarray*} 
 \Delta_{f} \large ({R^{-1}}{|Rm|}  \large) & = & R^{-1} \Delta_{f} |Rm| +|Rm|\Delta_{f} (R^{-1})+2\na (R^{-1}|Rm|R)\cdot\na (R^{-1})\\
& \ge & -c R^{-1} {|Rm|^2}  +  2R^{-1} {|Rm|}  |\na \ln R|^2 -2\na ({R^{-1}} {|Rm|})\cdot \na \ln R \\
& &   -2 {R^{-1}} {|Rm|}  |\na \ln R|^2\\\
& \ge &  -c {R^{-1}} |Rm|^2   -2\na ({R^{-1}}{|Rm|})\cdot \na \ln R .
\end{eqnarray*}  

Thus,
\begin{equation}
\Delta_{f-2\ln R} \left(\frac {|Rm|} {R^{}} \right)  \geq -c \frac {|Rm|^2}  {R^{}}, 
\end{equation}
where $c$ is a universal constant.  

Similarly,  by Lemma 2.3 and direct computations, we have
\begin{eqnarray*} 
 \Delta_{f} \left(\frac {|Rc|^2} {R^{2}} \right) & = & R^{-2} \Delta_{f} (|Rc|^2) +|Rc|^2\Delta_{f} (R^{-2})+2\na |Rc|^2\cdot\na (R^{-2})\\
& \ge & 2 \frac {|\na Rc|^2} {R^{2}} -C \frac {|Rc|^2} {R^{2}} +2\frac {|Rc|^2} {R^{2}} \left(2 \frac {|Rc|^2} {R}+3|\na \ln R|^2\right)\\
&  & -2\na \left(\frac {|Rc|^2} {R^{2}}\right)\cdot \na \ln R -4 \frac {|Rc|^2} {R^{2}} |\na \ln R|^2-4\frac {|Rc|} {R^{}}\frac {|\na Rc|} {R^{}} |\na \ln R|\\
& \ge & 2 \frac {|\na Rc|^2} {R^{2}} -C \frac {|Rc|^2} {R^{2}} 
 +2\frac {|Rc|^2} {R^{2}}|\na \ln R|^2 -4\frac {|Rc|} {R^{}}\frac {|\na Rc|} {R^{}}  |\na \ln R|
\\
& & 
-2\na \left(\frac {|Rc|^2} {R^{2}}\right)\cdot \na \ln R \\
& \ge &  \frac {|\na Rc|^2} {R^{2}} -C \frac {|Rc|^2} {R^{2}} 
 -2\frac {|Rc|^2} {R^{2}}|\na \ln R|^2 -2\na \left(\frac {|Rc|^2} {R^{2}}\right)\cdot \na \ln R ,
\end{eqnarray*}  
where we have used  the fact that 
\[  \frac {|\na Rc|^2} {R^{2}} +4\frac {|Rc|^2} {R^{2}}|\na \ln R|^2 \ge 4\frac {|Rc|} {R^{}}\frac {|\na Rc|} {R^{}} |\na \ln R|.  \]

On the other hand, from (5.7) we observe that  
\begin{equation*}
|Rc|^2\leq R^2\le 1   \quad \text{and} \quad |\na F|^2 \le 1 . 
\end{equation*}
Hence, 
\begin{equation}
|\na \ln R|\leq 2 R^{-1} |Rc| |\na F|\leq 2 .  
\end{equation}
Therefore, 
\begin{equation}
\Delta_{f-2\ln R} \left(\frac {|Rc|^2} {R^{2}} \right)\ge  \frac {|\na Rc|^2} {R^{2}}  -C . 
\end{equation}
Combining (5.10) and (5.12), we get 
\[\Delta_{f-2\ln R} \left(\frac {|Rm|} {R^{}} +\lambda \frac {|Rc|^2} {R^{2}} \right)\ge \lambda \frac {|\na Rc|^2} {R^{2}}-c \frac {|Rm|^2}  {R^{}}-\lambda C. \]

Since $Rc>0$,  $\na R=-2 Rc(\na F, \cdot)$, and $\na F\neq 0$ on $M\backslash \{x_0\}$, the scalar curvature $R$ is strictly decreasing along $\na F$ direction.  Thus, it follows from $R+ |\na F|^2=1$ that $|\na F|\ge c$ for some constant $c>0$ outside the compact set $\{F\ge r_0\}$ for some $r_0>0$ sufficiently large. 

Now,  by applying Lemma 2.5 and choosing $\lambda$ suitably large, the function  
$$v:=\frac {|Rm|} {R^{}} +\lambda \frac {|Rc|^2} {R^{2}} $$ satisfies the differential inequality 
\[\Delta_{f-2\ln R} (v)\geq v^2-C.\]
Thus, a standard maximum principle argument implies $v \le C$, hence 
\[ |Rm| \le CR \quad \text{on} \ M .\]
\end{proof}

Finally, we can also prove estimates for the covariant derivative $\na Rm$ of the curvature tensor in both Theorem 5.1 and Theorem 5.2. 
\begin{theorem}
Let $(M^4, g, f)$ be a $4$-dimensional complete noncompact non Ricci-flat gradient steady Ricci soliton with 
either 

\smallskip
(a) $\lim_{x\to \infty} R(x)=0$, or 

\smallskip
(b) $Rc > 0$ and  $R$ attains its maximum at some point $x_0\in M$. 

\smallskip
\noindent Then, there exists some constant $C > 0$, depending on $A_0$ such that 
\[ |\na Rm| \leq  CR  \quad \mbox{on} \ M. \]
\end{theorem}

\begin{proof}  By Lemma 2.4 and (5.2) or (5.6), we have 
\[\Delta_{f} |\na Rm|\ge -c|Rm||\na Rm|\ge -C |\na Rm|.\]
Then, by (5.9) and direct computations, 
\begin{eqnarray*} 
\Delta_{f} (R^{-1} {|\na Rm|}) & = & R^{-1} \Delta_{f} |\na Rm| +|\na Rm|\Delta_{f} (R^{-1})+2\na (R^{-1}|\na Rm|R)\cdot\na (R^{-1})\\
& \ge & -C R^{-1} |\na Rm| +  2R^{-1}|\na Rm|  (R^{-1}|Rc|^2 + |\na \ln R|^2)\\
& &  -2 \na (R^{-1}{|\na Rm|})\cdot \na \ln R -2 R^{-1} {|\na Rm|} \cdot |\na \ln R|^2\\
& \ge & -C R^{-1} |\na Rm| 
-2 \na (R^{-1}{|\na Rm|})\cdot \na \ln R. 
\end{eqnarray*}
Thus, 
\begin{equation}
\Delta_{f-2\ln R} (R^{-1} {|\na Rm|}) \ge -C R^{-1} |\na Rm| .
\end{equation}
Also, by Lemma 2.4 and  (5.2) or (5.6), we have 
\[\Delta_{f} |Rm|^2\ge 2|\na Rm|^2-C|Rm|^3\ge 2|\na Rm|^2-C|Rm|^2.\]
Then, by (5.9) and direct computations, we obtain
\begin{eqnarray*} 
 \Delta_{f} \large (\frac {|Rm|^2} {R^{2}} \large) & = & R^{-2} \Delta_{f} (|Rm|^2) +|Rm|^2\Delta_{f} (R^{-2})+2\na |Rm|^2\cdot\na (R^{-2})\\
& \ge & 2 \frac {|\na Rm|^2} {R^{2}} -C \frac {|Rm|^2} {R^{2}} +2\frac {|Rm|^2} {R^{2}} \left( \frac {|Rc|^2} {R}+3\frac{|\na  R|^2}{R^2}\right)\\
&  & -2\na \left(\frac {|Rm|^2} {R^{2}}\right)\cdot \frac {\na R}{R} -4 \frac {|Rm|^2} {R^{2}} \frac {|\na R|^2}{R^2}-4\frac {|Rm|} {R^{}}\frac {|\na Rm|} {R^{}} {|\na \ln R|}\\
& \ge & 2 \frac {|\na Rm|^2} {R^{2}} -C \frac {|Rm|^2} {R^{2}} -2\na \left(\frac {|Rm|^2} {R^{2}}\right)\cdot \frac {\na R}{R} \\
& & +2\frac {|Rm|^2} {R^{2}}|\na  \ln R|^2 -4\frac {|Rm|} {R^{}}\frac {|\na Rm|} {R^{}} |\na \ln R|\\
& \ge &  \frac {|\na Rm|^2} {R^{2}}  -C \frac {|Rm|^2} {R^{2}} -2\frac {|Rm|^2} {R^{2}}|\na  \ln R|^2-2\na \left(\frac {|Rm|^2} {R^{2}}\right)\cdot \na \ln R.
\end{eqnarray*}  
Since $|Rc|^2 \le CR^2$ and $|\na F|^2 \le 1$, we have 
\begin{equation}
|\na \ln R|\leq 2 R^{-1} |Rc| |\na F|\leq C .  
\end{equation}
It follows that  
\begin{equation}
\Delta_{f-2\ln R} \left(\frac {|Rm|^2} {R^{2}} \right)  \ge  \frac {|\na Rm|^2} {R^{2}}  -C \frac {|Rm|^2} {R^{2}}  .
\end{equation}
Combining (5.13) and (5.15), and setting 
\[ w=\frac {|\na Rm|} {R} + \frac {|Rm|^2} {R^{2}},\] we get
 \begin{eqnarray*} 
\Delta_{f-2a\ln R} (w) & \ge  &  \frac {|\na Rm|^2} {R^{2}}  -2C_1 \frac {|\na Rm|}{R }-C_2 \frac {|Rm|^2} {R^{2}} \\
& \ge & \left(\frac {|\na Rm|} {R^{}}+  \frac {|Rm|^2} {R^{2}}\right)^2 - C \frac {|Rm|^4} {R^{4}}
-C_1\frac {|\na Rm|}{R} \\
& & -C_2 \frac {|Rm|^2} {R^{2}}\\
& \ge & w^2 - C_1 w-C_2,
\end{eqnarray*} 
where in the last inequality we have used Theorem 5.1 and Theorem 5.2. 
Now,  it follows from a standard maximum principle argument that $w\le C$, hence
\[ |\na Rm|\le CR \quad \text{on}  \ M . \] 
\end{proof}

\section {Further Discussion} 

In this last section, we draw some conclusions and raise some open questions.

\medskip

{\bf Conclusion 1}. The scalar curvature lower bound $R\ge C/f$ of Chow-Lu-Yang \cite{CLY} played a crucial rule in the curvature estimates $|Rm|\le CR$ and $|\na Rm|\le CR$ for shrinkers by Munteanu and Wang \cite{MW}, especially in {\bf Step 5-Step 9} of the sketched proof.  While lacking such a scalar curvature lower bound in the expanding case, by picking the curvature quantities suitably, we still managed to prove the almost sharp curvature estimate for $Rm$ and $\na Rm$ for 4-dimensional gradient expanding solitons with nonnegative Ricci curvature $Rc\ge 0$. 
On the other hand, if there is any polynomial lower bound for $R$, then we would achieve the sharp curvature estimate $|Rm|\le CR$ in the expanding case. 

\medskip
{\bf Question 1}. Suppose $(M^n, g, f)$ is a complete noncompact non-flat gradient expanding soliton with  nonnegative Ricci curvature $Rc\ge 0$. Does its scalar curvature $R$ have at most polynomial decay $R\ge C/F^d$ for some $d\ge 2$ at infinity?

\begin{remark} By comparing the proofs of Theorem 3.1 and Theorem 4.3, one see that the scalar curvature lower bound $R\ge C/f$ is not really needed in {\bf Step 3} (for Theorem 3.1) if one makes use of the bound  $|Rm|\le C$ as in {\bf Step III} (for Theorem 4.3).  
\end{remark}

\begin{remark} 
Recently, P.-Y. Chan \cite{Chan2} has proved an exponential decay scalar curvature lower bound for expanding solitons with positive scalar curvature  $R> 0$ and proper $f$: Suppose $(M^n, g, f)$, $n\ge 2$,  is any $n$-dimensional complete noncompact gradient expanding Ricci soliton with positive scalar curvature  $R> 0$ and proper potential function $F$. Then
\[ R\ge CF^{1-\frac n2}e^{-F} \quad on \ M^n\]
for some constant $C>0$.  

Moreover, based on the constructions of Deruelle \cite{Der16, Der17}, very recently Chan and Zhu \cite{Chan3} have exhibited an example of $3$-dimensional complete noncompact asymptotically conical gradient expanding soliton with nonnegative curvature operator $Rm\ge 0$ such that 
\[ \liminf_{r\to \infty} \ F |Rm| =0 \quad  \mbox{and} \ \limsup_{r\to \infty} \ F |Rm| <\infty . \] 
\end{remark}

\medskip

{\bf Conclusion 2}. Gradient estimates for scalar curvature $R$ is quite essential whenever we involve the operator $\Delta_{f-2\ln R}$ in the maximum principle argument. 

In the shrinking case,  in Step 7 and Step 9, the gradient estimate $|\na \ln R|^2 \le C\ln (f+2)$ in {\bf Step 6} is used to derive the sharp $Rc$ estimate $|Rc| \le CR$ (hence the sharp $Rm$ estimate $|Rm|\le CR$) and the estimate $|\na Rm|\le CR$.  

Similarly, for expanders with $Rc\ge 0$, the gradient estimate 
$|\na \ln R| \le 2 |\na F|$ in (4.19) is responsible for obtaining the almost sharp $Rm$ estimate and the estimate on 
$|\na Rm|$. 

In the steady case, we have the sharp gradient estimate $|\na \ln R|\le C$, which made the sharp estimates $|Rm|\le CR$ and $|\na Rm|\le CR$ possible. 

\medskip
{\bf Question 2}. Suppose $(M^4, g, f)$ is a complete noncompact gradient expanding Ricci soliton with $0<R \le R_0$ and proper $f$. Does the gradient estimate $|\na \ln R|\le C |\na F|$, or more precisely, 
\begin{equation}
R^{-1} { Rc(\na F, \na F) }  \le C,
\end{equation}
hold for some constant $C>0$ ? 

\begin{remark} 
If (6.1) holds for $(M^4, g, f)$, then one would be able to improve the conclusions in Theorem 4.3 to  $|Rm|^2\le CR$ and $|\na Rm|^2\le CR$. 
\end{remark} 

Also, in the shrinking case, the estimate $|\na Rm | \le C R$ implies the sharp gradient estimate 
\begin{equation}
|\na \ln R| \le C \quad \text{on} \ M .
\end{equation}
In the expanding case, from $|\na Rm| \le C R^{a} $ for any $0<a<1$ we also have 
\begin{equation}
|\na \ln R| \le C_aR^{a-1} \quad \text{on} \ M .
\end{equation}

\medskip
 {\bf Question 3}. Does gradient estimate (6.2) for $R$ hold for 4-dimensional gradient expanding Ricci solitons with nonnegative Ricci curvature? Is it possible to prove (6.2) directly in the shrinking case? 

\begin{remark}  If the answer to Question 3 is affirmative, then one would be able to improve the conclusions in Theorem 4.1 to  $|Rm|\le CR$ and $|\na Rm|\le CR$ without any extra assumption.
\end{remark} 

\medskip

{\bf Conclusion 3}. Some key steps, e.g., {\bf Step 5} and {\bf Step 9}, in the proof of the shrinking case do not seem to work for expanders (even if one has the improved gradient estimate $|\na \ln R|^2\le C_1 (\ln F+C_2)$ for expanders). This seems to be mainly due to the fundamentally distinct nature between shrinkers and expanders which is partly reflected in the different curvature differential identities: 
\[\Delta_f Rm= Rm  + Rm\ast Rm  \quad \text{versus} \quad \Delta_f Rm= -Rm  + Rm\ast Rm .\]
 If one draws a comparison to the heat equation, in the shrinking case $\Delta_f:=\Delta -\na f\cdot \na $ behaves like a heat operator, where the $\na f$ term corresponds to the time derivative, whereas in the expanding case $\Delta_f$ behaves like a backward heat operator. To some  extent, this seems to explain why the same arguments do not work in the expanding case. In any case, 
it  remains a challenge to prove the sharp estimate $|Rm|\le CR$ for 4-dimensional expanders (with $Rc\ge 0$).

\medskip
Finally, we would like to point out that recently Chow et al \cite{Chow et al} have proved  the following curvature estimates for complete gradient shrinking and steady solitons that are singularity models.  Here, by gradient shrinking or steady {\it Ricci soliton singularity model}, it means that $(M^4, g, f)$ arises as a limit of parabolic dilations at a finite time singularity of the Ricci flow ${\bar g}(t)$ on a closed oriented $4$-manifold  $\bar{M}^4$.

\begin{theorem} {\bf (Chow-Friedman-Shin-Zhang \cite{Chow et al} )}. Any $4$-dimensional gradient shrinking Ricci soliton singularity model $(M^4, g, f)$ must have (at most) quadratic curvature growth, i.e., there exists some $C>0$  such that
\[ |Rm| \le C (1+ r(x))^2 \quad \text{on} \ M ,\]
where $r(x)$ is the distance function on $M$ from some fixed base point. 
\end{theorem}

\begin{remark} Since $R+|\na f|^2=f$ and $R>0$, it follows from Proposition 3.1 that 
\begin{equation}
R\le \frac 1 4 (r(x)+c_2)^2
\end{equation}
for any complete noncompact gradient shrinking Ricci soliton.  
\end{remark}

\begin{theorem} {\bf (Chow-Friedman-Shin-Zhang \cite{Chow et al} )}. Any $4$-dimensional gradient steady Ricci soliton singularity model $(M^4, g, f)$ must have bounded curvature, i.e., there exists some $C>0$  such that
\[ |Rm| \le C \quad \mbox{on} \ M .\]
\end{theorem} 

Naturally, one can ask 

\medskip
 {\bf Question 4\footnote{\ \!Note Added in Proof: In the steady soliton case, Question 4 has been answered affirmatively by Chan-Ma-Zhang \cite{Chan4} very recently. Namely, $|Rm|\le CR$ holds in Theorem 6.2.}}. Does $|Rm|\le CR$ hold in Theorem 6.1 and Theorem 6.2?

\end{document}